\documentclass[12pt,a4paper]{amsart}
\usepackage[english]{babel}
\usepackage{amssymb}
\usepackage{graphicx}
\usepackage{geometry}
\usepackage{color}
\usepackage{epsfig}
\usepackage{subfigure}
\usepackage{tikz}
\usetikzlibrary{positioning}
\usepackage{caption}
\usepackage{hyperref}  
\hypersetup{colorlinks=true, linkcolor=blue, anchorcolor=blue, 
citecolor=red, filecolor=blue, menucolor=blue,
urlcolor=blue}

\numberwithin{equation}{section}
\newtheorem{theorem}{Theorem}[section]

\newtheorem{proposition}[theorem]{Proposition}
\newtheorem{lemma}[theorem]{Lemma}
\theoremstyle{remark}
\newtheorem{remark}{Remark}[section]

\theoremstyle{definition}

\newcommand{\R}{{\mathbb R}}
 
\newcommand{\Z}{{\mathbb Z}}
\newcommand{\C}{{\mathbb C}}

\newcommand{\PP}{\mathcal{P}_k}

\newcommand{\sol}{e^{-it\mathcal{D}_\nu}u_0}

\newcommand{\D}{\mathcal{D}_\nu}

\def\XXint#1#2#3{{\setbox0=\hbox{$#1{#2#3}{\int}$ }
\vcenter{\hbox{$#2#3$ }}\kern-.58\wd0}}

\makeatletter
\@namedef{subjclassname@2020}{
Mathematics Subject Classification}
\makeatother
\newcommand{\msc}[1]{\href{https://mathscinet.ams.org/mathscinet/search/mscdoc.html?code=#1}{#1}}

\begin{document}

\title
[Asymptotic  for the wave functions of the Dirac-Coulomb operator]
{Asymptotic estimates for the wave functions of the Dirac-Coulomb operator and applications}
\begin{abstract}
In this paper we prove some uniform asymptotic estimates for confluent hypergeometric functions making use of the steepest-descent method. As an application, we obtain Strichartz estimates that are $L^2$-averaged over  angular direction for the massless Dirac-Coulomb equation in $3D$.
\end{abstract}
\keywords{Dirac-Coulomb equation; Strichartz estimates; steepest descent method}
\subjclass[2020]{\msc{35Q41}}

\author{Federico Cacciafesta}
\address{Federico Cacciafesta: 
Dipartimento di Matematica, Universit\'a degli studi di Padova, Via Trieste, 63, 35131 Padova PD, Italy}
\email{cacciafe@math.unipd.it}

\author{\'Eric S\'er\'e }
\address{\'Eric S\'er\'e: CEREMADE, UMR CNRS 7534, Universit\'e Paris-Dauphine, PSL Research University,
Pl. de Lattre de Tassigny,
75775 Paris Cedex 16, France}
\email{sere@ceremade.dauphine.fr}

\author{Junyong Zhang}
\address{Junyong Zhang:
Department of Mathematics, Beijing Institute of Technology, Beijing 100081; 
}
\email{zhang\_junyong@bit.edu.cn
}

\maketitle

\tableofcontents

\section{Introduction}

In this paper we study the flow of the $3$d massless Dirac-Coulomb equation, that is the following first-order system
\begin{equation}\label{diraccoul}
\begin{cases}
\displaystyle
 i\partial_tu=\mathcal{D}_\nu u\,,\quad u(t,x):\mathbb{R}_t\times\mathbb{R}_x^3\rightarrow\mathbb{C}^{4}\\
u(0,x)=u_0(x)
\end{cases}
\end{equation}
where $$\mathcal{D}_\nu=\mathcal{D}-\frac{\nu}{|x|}\; ,\ \ 
\mathcal{D}=-i\displaystyle\sum_{k=1}^3\alpha_k\partial_k=-i(\alpha\cdot\nabla)
$$ 
\begin{equation}
\alpha_k=\left(\begin{array}{cc}0 & \sigma_k \\\sigma_k & 0\end{array}\right),\quad k=1,2,3.
\end{equation}
Here, $\sigma_k\; (k=1,2,3)$ are the Pauli matrices
\begin{equation}\label{sigma}
\sigma_1=\left(\begin{array}{cc}0 & 1 \\1 & 0\end{array}\right),\quad
\sigma_2=\left(\begin{array}{cc}0 &-i \\i & 0\end{array}\right),\quad
\sigma_3=\left(\begin{array}{cc}1 & 0\\0 & -1\end{array}\right)\,.
\end{equation}

It is now well understood that the charge needs to satisfy the condition $|\nu|\leq 1$ in order for the operator $\mathcal{D}_\nu$ to have a physically meaningful self-adjoint realization in the Hilbert space $L^2(\R^3,\C^4)$. To be more precise, $\mathcal{D}_\nu$ defined on the minimal domain $\mathcal C^\infty_c(\R^3\setminus\{0\},\C^4)$ happens to be essentially self-adjoint in the range $|\nu|\leq \frac{\sqrt3}2$, and admits a distinguished self-adjoint extension in the range $\frac{\sqrt3}2<|\nu|\leq1$ (see \cite{estlewser} and references therein).

From the point of view of dispersive analysis, system \eqref{diraccoul} is quite delicate, as indeed the Coulomb potential is {\em critical} with respect to the scaling of the massless Dirac equation and, as it is well known, scaling critical perturbations can be very difficult to handle, as they typically rule out perturbative arguments. Dispersive estimates for the Dirac equation with subcritical potential perturbations have been proved e.g. in \cite{erdgreen}, \cite{erdgreen2}, \cite{danfan}, \cite{bousdanfan}, \cite{cacdan}. We refer to \cite{cacser} for a short overview of the topic; in that paper the following local smoothing estimates were proved for the solutions to \eqref{diraccoul}
\begin{equation}\label{mor}
\displaystyle
\left\||x|^{-\alpha}\left|\mathcal{D}_\nu\right|^{1/2-\alpha} u\right\|_{L^2_tL^2_x}\leq C \|u_0\|_{L^2}
\end{equation}
for any $u_0\in L^2$ and $\frac12<\alpha<\sqrt{1-\nu^2}+\frac12$, but this estimate is not strong enough to obtain Strichartz estimates by relying on the standard Duhamel argument. In the spirit of \cite{burq1}, the main idea in the proof in \cite{cacser} relies on the use of partial wave decomposition and on the construction of a ``relativistic (or distorted) Hankel transform" (these tools will be recalled in subsection \ref{pwdsub}). This transform allows to ``diagonalize" the problem on a fixed spherical level, and this allows to obtain local smoothing estimates after a careful analysis of some integrals involving suitable products of generalized eigenfunctions. This same strategy has been later developed in \cite{cacfan} to obtain similar estimates for the Dirac equation in an Aharonov-Bohm (AB) field. We should stress the fact that the Dirac-Coulomb equation turns out to be significantly more difficult than the AB case, due to the much more complicated structure of the eigenfunctions, as we will see later. Notice that for all these results it is crucial to assume the mass to be zero, as the presence of a positive mass would destroy the scaling of the system. 

\medskip

In order to prove Strichartz estimates with angular regularity for the solutions to \eqref{diraccoul}, we are going to adapt the machinery developed in \cite{miao} in another context: the wave equation with inverse square potential (see also \cite{caczhaab} for an adaptation to the Dirac equation in the (AB) field). The arguments can be summarized in the following steps:

\begin{enumerate}
\item Use partial wave decomposition to reduce the dynamics to a radial system;
\item Use the relativistic Hankel transform built in \cite{cacser} to obtain a ``nice" integral representation of the solution based on the generalized eigenfunctions of the operator $\mathcal{D}_\nu$;
\item Prove suitable pointwise estimates on the generalized eigenfunctions;
\item Deduce Strichartz estimates.
\end{enumerate}

Steps $(1)$ and $(2)$ have already been dealt with: the former is completely classical, while the latter has been introduced in \cite{cacser}, so we will only need to review them in section \ref{preliminaries}. Step $(3)$ is the one that requires most of the technical work. Indeed, one of the main ingredients in the proof of \cite{miao} consists in providing suitable estimates on the generalized eigenfunctions of the operator $-\Delta+\frac{a}{|x|^2}$, which are known to be standard Bessel functions, and for them the following precise estimate is available for $\lambda\gg1$:
\begin{eqnarray}\label{estbes}
|J_\lambda(r)|\leq C\times
\begin{cases}
e^{-D\lambda},\qquad\qquad\qquad\qquad\qquad\; 0<r\leq \lambda/2,\\
\lambda^{-1/4}(|r-\lambda|+\lambda^{1/3})^{-1/4},\qquad \lambda/2<r\leq 2\lambda,\\
r^{-1/2},\qquad\qquad\qquad\qquad\qquad\; 2\lambda<r
\end{cases}
\end{eqnarray}
 for some positive constants $C$ and $D$ independent of $r$ and $\lambda$ (for this estimate see e.g. \cite{cor}-\cite{barcor}-\cite{stempak} ).
In the Dirac-Coulomb case on the other hand, the expression of the eigenfunctions is much more complicated and involves confluent hypergeometric (or Whittaker) functions (we postpone the overview of the spectral theory of the Dirac-Coulomb operator to section \ref{preliminaries}); to the very best of our knowledge, an analog of estimate \eqref{estbes} has not been proved for the confluent hypergeometric functions, and we will therefore need to provide one.

Finally, for what concerns step $(4)$ of the strategy above, we will follow the argument from \cite{miao}, that can be again decomposed in the following sub-steps:

\begin{enumerate}
\item Prove Strichartz estimates on a fixed angular level and with unit frequency;
\item Deduce Strichartz estimates for the complete dynamics using a scaling argument and a dyadic decomposition.
\end{enumerate}

The first result of this manuscript, which is of independent interest, is thus the analog of estimates \eqref{estbes} for the generalized eigenfunctions $\psi_k(\rho)$ of the massless Dirac-Coulomb operator $\mathcal{D}_\nu$. A precise definition of these functions will be given in section \ref{preliminaries}.
We will prove the following

\begin{theorem}\label{esteigen}
Given $\nu\in [-1,1]$ and $k\in\Z^*$, let $\gamma=\sqrt{k^2-\nu^2}$ and consider the generalized eigenfunction $\psi_{k}=\left(\begin{array}{cc}F_{k}\\\:G_{k}
\end{array}\right)$ of $\mathcal{D}_\nu$ with eigenvalue $E=1$, given by formula \eqref{eigen1} of section \ref{preliminaries}. Then there exist positive constants $C\,,\,D$ independent of $k\,,\,\nu$ such that the following pointwise estimate holds for all $\rho\in \R\backslash\{0\}$: 
\begin{equation}\label{esteigformula}
 j_{0,k}(\rho) + j_{1,k}(\rho) \leq C
\begin{cases}
(\min\{|\rho|/2\,,\,1\})^{\gamma-1}e^{-D\vert k \vert}\,,\qquad\quad\, 0<|\rho|\leq \max\{\vert k \vert/2\,,\,2\},\\
\vert k \vert^{-\frac34}\big(|\,\vert k \vert-|\rho|\,|+\vert k \vert^{\frac13}\big)^{-\frac14},\quad\quad \frac{\vert k \vert}2\leq |\rho|\leq 2\vert k \vert,
\\
|\rho|^{-1}
,\qquad\qquad\qquad\qquad\qquad\quad\quad\quad\;\,\, |\rho|\geq 2\vert k \vert,
\end{cases}
\end{equation}
where
\begin{equation}\label{j}
j_{0,k}(\rho):=|\psi_k(\rho)|\ \hbox{ and }\  j_{1,k}(\rho):=|\psi'_k(\rho)-(\gamma-1)\rho^{-1}\psi_k(\rho)|\,.
\end{equation}
\medskip

\end{theorem}

\begin{remark}
As a direct consequence of \eqref{esteigformula}, one gets an estimate for $|\psi_k'(\rho)|\,$ (with a possibly larger $C$ and a smaller $D>0\,$, both independent of $k$, $\nu$):
\begin{equation}\label{esteigformula'}
 |\psi_k'| \leq C
\begin{cases}
(\min\{|\rho|/2\,,\,1\})^{\gamma-2}e^{-D\vert k \vert}\,,\qquad\quad\, 0<|\rho|\leq \max\{\vert k \vert/2\,,\,2\},\\
\vert k \vert^{-\frac34}\big(|\,\vert k \vert-|\rho|\,|+\vert k \vert^{\frac13}\big)^{-\frac14},\quad\quad \frac{\vert k \vert}2\leq |\rho|\leq 2\vert k \vert,
\\
|\rho|^{-1}
,\qquad\qquad\qquad\qquad\qquad\qquad\quad\;\;\,\, |\rho|\geq 2\vert k \vert.
\end{cases}
\end{equation}
\end{remark}

\begin{remark}\label{remark1}
Our proof is different from the one of estimate \eqref{estbes} as given in \cite{cor}-\cite{barcor}. In that case an integral formula for Bessel functions with an oscillating integrand was used, allowing an application of the Van der Corput method. The integral formulation for $\psi_k(\rho)$ is more difficult to deal with, and we will rely on the steepest descent method to deal with large values of $\vert\rho\vert$ and $\vert k \vert$. For the reader's convenience, we give a brief description of this method in section \ref{method}. Let us just mention here that the steepest descent method, in its simplest version, typically provides exact asymptotic formulas for integrals depending on one large parameter. We are not exactly interested in such formulas in the present work: instead, we look for uniform estimates valid for all values of the three parameters $\rho,\,k,\,\nu$ and having the same accuracy as the known estimates on Bessel functions. This goal is achieved: in the limit $\nu\rightarrow 0$, the functions $\psi_k$ reduce to the Bessel functions (with proper weights and coefficients), and the estimates proved in Theorem 1.1 retrieve estimates (1.5). Let us also point out that one could write exact asymptotic formulas for $\psi_k(\rho)$ valid when $\vert\rho\vert\to\infty$ with $\frac{\vert k\vert }{\vert\rho\vert}$ fixed, by modifying slightly our arguments. One would then see that our estimates of polynomial decay are optimal up to a multiplicative constant when $\frac{\vert k\vert }{\vert\rho\vert}\leq 1 $. However, if $\vert k\vert > \rho$ and $\frac{\vert k\vert }{\vert\rho\vert}$ stays away from $1$, then $\psi_k(\rho)$ decreases exponentially, as mentioned in Subcase (1.a) of our proof. The statement of Theorem 1.1 is not optimal in this sector: the rate of exponential decay is estimated roughly in the sector $\vert\rho\vert<\frac{\vert k\vert}{2}$ and does not even appear in the sector $\frac{\vert k\vert}{2}\leq \vert\rho\vert< \vert k\vert$. But better bounds in the domain $\vert k\vert > \rho$ would not improve our Strichartz estimates, that are the final purpose of this paper.
\end{remark}

As an application of Theorem \ref{esteigen}, by following the aforementioned strategy, we are able to obtain Strichartz estimates that are $L^2$-averaged over angular direction for solutions to \eqref{diraccoul} for general initial data.  Before stating the result, let us mention that we will use standard notations for Lebesgue and Sobolev spaces; if not specifically indicated, the norms will be intended on the whole space (i.e. $L^p_t=L^p_t(\R)$ and $L^q_x=L^q_x(\R^3)$), and we will systematically omit the dimension on the target space. We will denote with $L^p_tL^q_x=L^p(\mathbb{R}_t; L^q(\mathbb{R}^3_x))$ the mixed space-time Strichartz spaces. Using the polar coordinates $x=r\omega$, $r\geq0$, $\omega\in S^2$, and given a measurable function $F=F(t,x)$ we will denote by
$$
\|F\|_{L^p_tL^q_{r^2dr}L^2_\omega}
:=
\left(\int_{\R}\left(\int_{0}^{+\infty}\left(\int_{ S^2}|F(t,r,\omega)|^2\,d\sigma\right)^{\frac q2}\,r^{2}dr\right)^{\frac pq}\,dt\right)^{\frac1p},
$$
being $d\sigma$ the surface measure on the sphere. We then have the following

\begin{theorem}[Strichartz estimates]\label{strichteo}
Let $|\nu|<\frac{\sqrt{15}}4$. For any $u_0\in \dot H^s$, the following Strichartz estimates hold
\begin{equation}\label{strinonend}
\|\sol\|_{L^p_tL^q_{r^2dr}L^2_\omega}\leq
C
\|u_0\|_{\dot{H}^{s}},
\end{equation}
provided that 
\begin{equation}\label{q-condbis}
p\geq2, \quad 4<q<\frac{3}{1-\sqrt{1-\nu^2}},\quad s=\frac32-\frac1p-\frac3q\,.
\end{equation}

\end{theorem}

\begin{remark}
The use of $L^p_{r^2dr}L^q_\omega$ spaces in order to obtain a "refined version" of Strichartz estimates is definitely not new (these are often referred to as {\em generalized Strichartz estimates)}; we mention at least \cite{Tao00}, in which the author retrieves the endpoint estimate for the $2D$ Schr\"odinger equation by averaging the solution in $L^2$ in the angular variable, then \cite{sterbenz} and \cite{mach} in which generalized estimates are obtained for the wave and the Dirac equation respectively. 
In particular, we should stress the fact that \eqref{strinonend} implies standard Strichartz estimates (i.e. in the spaces $L^p_t L^q_x$) for "radial" initial data (the precise meaning of "radial" in our context will be explained in the next section). Also, we mention that estimates \eqref{strinonend} might be used to prove well-posedness for some nonlinear models in a more or less standard way (by assuming "radial symmetry" on the initial data or by requiring additional angular regularity): we refer e.g. to \cite{mach}, \cite{cac1}, \cite{cacdan}.
\end{remark}

\begin{remark}
Let us comment on the constraints on the parameters $\nu$, $p$, $q$ in \eqref{q-condbis}. First of all, note that from the assumptions of the Theorem, the regularity parameter $s$ must lie in $(\frac14,\frac12+\sqrt{1-\nu^2})$. For what concerns the Strichartz pairs, we should stress the fact that we do not recover the full optimal range (compare e.g. with \cite{jiang}): in fact, our result could be slightly improved by sharpening our strategy in some steps at the price of additional technicalities (see next remark). Also, we should mention that condition \eqref{q-condbis} can be significantly relaxed by requiring some structure on the initial data $u_0$, that is to be "orthogonal to the first partial wave subspaces" (see next section for the definition). In particular, this assumption allows to remove the upper bound on $q$ in \eqref{q-condbis} (and thus the necessary condition $|\nu|<\frac{\sqrt{15}}4$). This fact will be further motivated towards the end of the paper (see formulas \eqref{4.12} and \eqref{condK}).
\end{remark}

\begin{remark}\label{strichartzrange}
Let us also briefly comment on the strategy of our proof: we will prove our Strichartz estimates on the "endpoint board line", that is for the choice $p=2$. Then, by interpolating with the standard $L^\infty_t H^s_x$ estimates, we will be able to cover the
range of parameters satisfying the assumptions of Theorem \ref{strichteo}. It is reasonable to expect that providing a direct proof in the case $L^p_t$ would allow to improve on the range of admissible exponents, but this would require a fair additional amount of technicalities that we prefer to avoid.
The upper bound $|\nu|<\frac{\sqrt{15}}4$ is a consequence of the necessary condition \eqref{q-condbis}, and it seems to have no physical meaning; it is again a byproduct of our proof. Notice anyway that $\frac{\sqrt{15}}4>\frac{\sqrt3}2$, and thus this range includes the set of charges that make the Dirac-Coulomb operator essentially self-adjoint.  \end{remark}

The structure of the paper is the following: section \ref{preliminaries} will be devoted to introduce the necessary setup, that is a review of the partial wave decomposition, of the spectral theory of the Dirac-Coulomb operator and of the method of the steepest descents. Section \ref{esteigesec} will be devoted to the proof of Theorem \ref{esteigen}: as we will see, several cases and sub-cases need to be considered in order to check every detail, and as a result the proof turns out to be quite long and slightly involved at some stages. Section \ref{strichsec} is devoted to the proof of Theorem \ref{strichteo}.

\medskip

\section{Preliminaries}\label{preliminaries}

We devote this section to review the necessary background needed for our main results.

\subsection{Partial wave decomposition and generalized Hankel transform}\label{pwdsub}

In this section we recall some classical facts already discussed in \cite{cacser} on the spectral properties of the Dirac-Coulomb operator in $3$d, together with the construction of the relativistic Hankel transform, that will play a crucial role in what follows. 

 The first main ingredient we need to introduce is the so called \emph{partial wave decomposition}, that essentially allows to reduce the action of the Dirac-Coulomb operator to (a sum of) radial operators with respect to some suitable decomposition. Let us give a brief overview of the construction: we refer to \cite{thaller} section 4.6 for detail.

 First of all, we use spherical coordinates to write
$$L^2(\mathbb{R}^3,\mathbb{C}^4)\cong L^2((0,\infty),r^2dr)\otimes L^2(S^2,\mathbb{C}^4)$$ 
with $S^2$ being the unit sphere. Then, we have the orthogonal decomposition on $S^2$:
$$L^2(S^2,\mathbb{C}^4)\cong \bigoplus_{k\in \mathbb{Z}^*}\bigoplus_{m\in \mathcal{I}_k} \mathfrak{h}_{k,m}\;.$$ 
Here, $\mathbb{Z}^*=\mathbb{Z}\backslash\{0\}$, $\mathcal{I}_k:=\{-\vert k\vert+1/2,-\vert k\vert+3/2,\cdots, \vert k\vert-1/2\}\subset \mathbb{Z}+1/2$ and each subspace $\mathfrak{h}_{k,m}$ is two-dimensional, with orthonormal basis
\begin{equation*}
\Xi_{k,m}^+=\left(\begin{array}{cc}i\, \Omega_{k,m} \\0\end{array}\right),\qquad
\Xi_{k,m}^-=\left(\begin{array}{cc}0 \\ \Omega_{-k,m}\end{array}\right)\;.
\end{equation*}
The functions $\Omega_{k,m}$ can be explicitly written in terms of standard spherical harmonics as
$$
\Omega_{k,m}=\frac{1}{\sqrt{|2k+1|}}\left(\begin{array}{cc}\sqrt{|k-m+1|}\,Y^{m-1/2}_{|k+1/2|-1/2} \\ {\rm sgn} (-k)\sqrt{|k+m+1|}\,Y^{m+1/2}_{|k+1/2|-1/2}\end{array}\right)\;.
$$ 
We thus have the unitary isomorphism $$L^2(\mathbb{R}^3,\mathbb{C}^4)\cong \bigoplus_{\substack{k\in \mathbb{Z}^*\\ m\in \mathcal{I}_k}} L^2((0,\infty),r^2dr)\otimes  \mathfrak{h}_{k,m}$$ given by the decomposition
\begin{equation}\label{explicitiso}
\Phi(x)=\sum_{k\in \mathbb{Z}^*}\sum_{m\in \mathcal{I}_k} f^+_{k,m}(r)\Xi_{k,m}^+(\theta,\phi)+f_{k,m}^-(r)\Xi_{k,m}^-(\theta,\phi)
\end{equation}
which holds for any $\Phi\in L^2(\R^3,\C^4)$.
The Dirac-Coulomb operator leaves invariant the partial wave subspaces $\mathcal{C}^\infty_c(0,\infty)\otimes \mathfrak{h}_{k,m}$ and its action on each column vector of radial functions $f_{k,m}=(f^+_{k,m},f^-_{k,m})^{\rm T}$ is given by the radial matrix
\begin{equation}\label{radred}
\mathcal{D}_{\nu,k}=\left(\begin{array}{cc}-\frac\nu{r} & -\frac{d}{dr}+\frac{1+k}r \\\frac{d}{dr}-\frac{1-k}r & -\frac\nu{r}\end{array}\right).
\end{equation}
This isomorphism allows for the following decomposition of the dynamics of the Dirac flow: for any $k\in \Z^*$ the choice of an initial condition as 
\begin{equation*}
u_{0,k,m}(x)=
f^+_{0,k,m}(r)\Xi^+_{0,k,m}(\theta,\phi)+
 f^-_{0,k,m}(r)
\Xi^-_{0,k,m}(\theta,\phi)
\end{equation*}
implies, by Stone's Theorem, that the propagator is given by
\begin{equation*}
 e^{-it\mathcal{D}_\nu}u_{0,k,m}=
 f^+_{k,m}(r,t)\Xi^+_{k,m}(\theta,\phi)+
f^-_{k,m}(r,t)
\Xi^-_{k,m}(\theta,\phi)
\end{equation*}
where
\begin{equation*}
\left(\begin{array}{cc} f^+_{k,m}(r,t)\\\
f^-_{k,m}(r,t)
\end{array}\right)=e^{-it\mathcal{D}_{\nu,k}}\left(\begin{array}{cc} f^+_{0,k,m}(r)\\
 f^-_{0,k,m}(r)
\end{array}\right).
\end{equation*}
In what follows, we will in fact use the shortened notation 
\begin{equation}\label{ppp3}
f \cdot \Xi_{ k,m}=f^+(r)\Xi^+_{k,m}(\theta,\phi)+
f^-(r)
\Xi^-_{k,m}(\theta,\phi)\,,\quad f(r)=(f^+(r), f^-(r))^{\mathrm{T}}.
\end{equation}
 
The second key tool we need is a suitable ``Hankel transform", that is a transformation that allows somehow to ``diagonalize" the action of the Dirac-Coulomb operator. Of course, one cannot use the standard Hankel transform: the generalized eigenstates are not Bessel functions, moreover positive and negative energy eigenstates are present and should be dealt with simultaneously. We thus define, for a fixed $k \in \mathbb{Z}^*$, a ``relativistic Hankel transform" of the form
\begin{equation}\label{H-1}
\mathcal{P}_k f(E)=\int_0^{+\infty}H_{k}(E r) f(r)r^{2}dr
\end{equation}
where $E\in (0,\infty)$ and, for any $\rho>0$, $H_{k}(\rho)=\left(\begin{array}{cc}F_{k}(\rho)\;\quad G_{k}(\rho)\\
F_{k}(-\rho)\;\quad G_{k}(-\rho)
\end{array}\right)\;.
$

The functions \begin{equation}\label{radeigpsi}
\psi_{k}(\pm E r)=
\left(\begin{array}{cc}F_{k}(\pm E r)\\\:G_{k}(\pm E r)
\end{array}\right)
\end{equation}
are the generalized eigenstates of the self-adjoint operator $\mathcal{D}_{\nu,k}$ with energies $\pm E$, so that
\begin{equation}\label{prophank}
\mathcal{P}_k \mathcal{D}_{\nu,k}=\mathrm{Diag}(E,-E)\mathcal{P}_k \,:
\end{equation}
in other words, the transform $\mathcal{P}_k$ ``diagonalizes" the operator $\mathcal{D}_{\nu,k}$ (see Proposition \ref{properties}).

\begin{remark}
The operator $\mathcal{D}_{\nu,k}$, its generalized eigenstates $\psi_{k}(\pm E r)$ and the transform $\mathcal{P}_k$ are independent of $m$.
\end{remark}

This construction suggests that the functions $\psi_{k}=\left(\begin{array}{cc}F_{k}\\\:G_{k}
\end{array}\right)$ play a crucial role, and most of the technical issues in our dispersive estimates will consist in proving suitable estimates for them. We therefore recall their precise definition, as given in  e.g. \cite{landlif2}, formulas (36.1)-(36-20): for fixed values of $k\in\Z^*$ and $\rho\in\mathbb{R}\backslash\{0\}$,
$F_k(\rho)$ and $G_k(\rho)$ are real and

\begin{equation}\label{eigen1}
(G_{k}+i F_{k})(\rho)=\frac{\sqrt{2}|\Gamma(\gamma+1+i\nu)|}{\Gamma(2\gamma+1)}e^{\pi\nu/2+i(\rho+\xi)}\vert 2\rho\vert^{\gamma-1}
{_1F_1}(\gamma-i\nu,2\gamma+1,-2i\rho)
\end{equation}
where $_1F_1(a,b,z)$ are \emph{confluent hypergeometric functions}, $\gamma=\sqrt{k^2-\nu^2}$  and $e^{-2i\xi}=\frac{\gamma-i\nu}{k}$ is a phase shift.

\medskip

One of the key tools of our strategy is represented by the following result, that has been proved in \cite{cacser}:
\begin{proposition}\label{properties}
For any $k\in\mathbb{Z}^*$ the following properties hold: 
\begin{enumerate}
\item
$\mathcal{P}_k$ is an $L^2$-isometry.
\item
$
\mathcal{P}_k\mathcal{D}_{\nu,k}=\sigma_3\Omega\mathcal{P}_k,
$
where $\Omega f(x):=|x| f(x)$. 
\item
The inverse transform of $\mathcal{P}_k$ is given by
\begin{equation}\label{H-1}
\mathcal{P}_k^{-1}f(r)=\int_0^{+\infty}H_{k}^{*}(E r) f(E)E^{2}dE
\end{equation}
where $H_{k}^{*}=\left(\begin{array}{cc}F_{k}(E r)\;\quad F_{k}(-E r)\\
G_{k}(E r)\;\quad G_{k}(-E r)
\end{array}\right)
$ (notice the misprint in formula \\(2.18) in \cite{cacser}).
\end{enumerate}
\end{proposition}

As a consequence of this Proposition, given a function $\displaystyle u_0=\sum_{\substack{k\in \mathbb{Z}^*\\ m\in \mathcal{I}_k}} f_{0,k,m}\cdot {\Xi}_{ k,m}$ we can decompose the solution to equation \eqref{diraccoul} as follows:
\begin{equation}\label{repsol}
\sol=\sum_{\substack{k\in \mathbb{Z}^*\\ m\in \mathcal{I}_k}}(e^{-it\mathcal{D}_{\nu, k}}f_{0,k,m})\cdot {\Xi}_{ k,m}
=\sum_{\substack{k\in \mathbb{Z}^*\\ m\in \mathcal{I}_k}}\PP^{-1}\left[e^{-itE\sigma_3}\big(\PP f_{0,k,m}\big)(E)\right]\cdot{\Xi}_{ k,m}.
\end{equation}
This decomposition will represent the essential starting point of our analysis.

\medskip
To conclude with this section, we provide a result of equivalence of the norms induced by the fractional powers of the Dirac-Coulomb operator with standard Sobolev spaces: more precisely, we have the following
\begin{lemma}\label{confnorm}
Let $|\nu|<1$. Then the following inequalities hold 
\medskip
\begin{itemize}
\item $\||\D|^s f\|_{L^2}\leq C_1 \| f\|_{\dot{H}^s}$ for any $s\in\left[0,1\right]$,
\medskip
\item $\| f\|_{\dot{H}^s}\leq C_2\||\D|^s f\|_{L^2}$ for any $s\in[0,\frac12+\sqrt{1-\nu^2}]$.
\end{itemize}
\medskip
Here, $C_1$ and $C_2$ depend on $\nu$ but remain bounded when $|\nu|$ stays away from $1$.
\end{lemma}
\begin{proof}
The first inequality is a consequence of the Hardy inequality.
The second one has been proved in \cite[Corollary 1.8]{frank}.

\end{proof}

\subsection{An integral representation for the generalized eigenfunctions of $\D$}

In order to prove estimate \eqref{esteigformula} it will be crucial to have an explicit integral representation for the functions $F_{k}$ and  $G_{k}$.
We resort on \emph{Whittaker functions}: we recall their definition 
$$
M_{\alpha,\mu}(z)=e^{-\frac12z}z^{\frac12+\mu}{_1F_1}(1/2+\mu-\alpha,1+2\mu,z)\,.
$$
We rely on the following integral representation for $M_{\alpha,\mu}$ (see \cite{abram}, p. 505)
\begin{equation}
\begin{split}
M_{\alpha,\mu}(z)&=\frac{\Gamma(1+2\mu)z^{\mu+\frac12}2^{-2\mu}}{\Gamma(1/2+\mu-\alpha)\Gamma(1/2+\mu+\alpha)}\int_{-1}^1e^{\frac12zt}(1+t)^{\mu-\frac12-\alpha}(1-t)^{\mu-\frac12+\alpha}dt
\end{split}
\end{equation}
and we take $\alpha=1/2+i\nu$ and $\mu=\gamma$, which gives
\begin{equation}\label{eigenintraprep}
(G_k+iF_k)(\rho) = \frac{|\Gamma(\gamma+1+i\nu)|}{\Gamma(\gamma+1+i\nu)}\frac{e^{\frac{\pi\nu}2}e^{i\xi}{\rho}^{\gamma-1}
}{2^{\gamma+1/2}\Gamma(\gamma-i\nu)}
\int_{-1}^1e^{-i\rho t}(1+t)^{\gamma-1-i\nu}(1-t)^{\gamma+i\nu} dt
\,.
\end{equation}
In what follows, we shall adopt the following compact notations.
Recalling the notations $\,j_{0,k}=|\psi_k|\,$ and $\,j_{1,k}=|\psi'_k-(\gamma-1)\rho^{-1}\psi_k|\,$ of Theorem \ref{esteigen},
we may write, for $\varepsilon\in\{0,1\}\,,$ $k\in \Z^*$, $\gamma=\sqrt{k^2-\nu^2}$ and $\rho\in \R^*\,$,
\begin{equation}\label{moduleigen}
j_{\varepsilon,k}(\rho) = \frac{e^{\frac{\pi\nu}2}\vert\rho\vert^{\gamma-1}
}{2^{\gamma+1/2}|\Gamma(\gamma-i\nu)|}
|I_{\varepsilon,\gamma,\rho}|
\end{equation}
with
\begin{equation}\label{integral}
I_{\varepsilon,\gamma,\rho}=\int_{-1}^1 e^{-i\rho t}t^\varepsilon(1+t)^{\gamma-1-i\nu}(1-t)^{\gamma+i\nu} dt\,.
\end{equation}

\subsection{The method of steepest descent.}\label{method}
 
In the proof of \eqref{esteigformula}, the main difficulty is to estimate $I_{\varepsilon,\gamma,\rho}$ when $\rho$ and $\gamma$ are both very large and $\gamma/\rho$ stays away from zero. To study this asymptotic regime, we resort to the \textit{method of steepest descent}, also called {\it saddle-point method}, that we now briefly recall in our context (for a general exposition and other examples we refer to \cite{erd}, \cite{dieu}, \cite{temme} and the references therein).\medskip

Since $\gamma-1$ is at least proportional to $\rho$, it is convenient to introduce the parameter $q=\frac{\gamma-1}{\rho}$. Then formula \eqref{integral} may be rewritten in the form
\begin{equation}\label{intC}
I_{\varepsilon,\gamma,\rho}=\int_{-1}^1 g_\varepsilon(t)e^{\rho h_q(t)}dt
\end{equation}
with
\begin{equation}\label{geps-hq}
g_\varepsilon(t)=t^\varepsilon(1+t)^{-i\nu}(1-t)^{1+i\nu}\ \ \hbox{ and }\ \ h_q(t)=-it+q\ln(1-t^2)\,.
\end{equation}

When $t\in (-1,1)$, $h_q(t)$ is neither real nor imaginary, so we cannot directly apply Laplace's method or the stationary phase. But $h_q$ and $g_\varepsilon$ can be analytically continued on $\Omega=\C\setminus ((-\infty,-1]\cup [1,\infty))$
and, by Cauchy's theorem, the value of $I_{\varepsilon,\gamma,\rho}$ is not modified if one deforms the interval of integration into a new oriented path $\Gamma_q$ in $\Omega$ having the same end points $-1$ and $1$. The ideal choice is a path made of ``steepest descent" curves of $\Re\{h_q(z)\}$, that is, curves tangent to the gradient of this function. Since $h_q$ satisfies the Cauchy-Riemann equation, $\Im\{h_q(z)\}$ is constant on these curves and they are separated from each other by saddle points of $\Re\{h_q(z)\}$ that are zeroes of the complex derivative $h_q'(z)$. Any maximizer of $\Re\{h_q(z)\}$ on $\Gamma_q$ must be such a saddle point. The asymptotic behaviour of the deformed integral  $\int_{\Gamma_q} g_\varepsilon(z)e^{\rho h_q(z)}dz$ for $\vert\rho\vert$ large is then found by Laplace's method and depends crucially on the behaviour of $h_q$ and $g_\varepsilon$ near the maximizers.
\medskip

There is an additional difficulty in our situation: the estimates we look for should be uniform in $q$, but the phase portrait of the vector field {\bf grad}$(\Re\{h_q\})$ varies a lot with $q$, even from a topological viewpoint. This forces us to split our study into several cases. Two ranges of values of $q$ are particularly problematic.\medskip

The first one is when $\vert q\vert$ is close to $1$. At $q=1$ one observes the coalescence of two saddle-points and this is why \eqref{esteigformula} only gives the estimate $\vert\psi_k(\pm k)\vert \leq C \vert k\vert^{-5/6}$ (which is optimal, as we will see in the proof) while, when $\vert\rho/k\vert $ stays away from $1$, the decay is faster. This type of degeneracy was studied in the general case by Chester, Friedman and Ursell \cite{chester}, who proved a uniform asymptotic formula involving the Airy function and its derivative. It would probably be possible to use their result followed by {\it a priori} bounds on the Airy function, but we chose to estimate directly our deformed integral after a careful choice of the integration path in which some parts of the steepest descent curves are replaced by piecewise affine approximations.\medskip

The second problematic case is when $0< \vert q\vert \ll 1$. Then the saddle points converge to $\pm 1$, that are branch points of $\ln(1-z^2)$. We solve this problem thanks to a suitable rescaling.

\section{Proof of Theorem \ref{esteigen}.}\label{esteigesec}

Several constants will appear throughout the proof, that will often be denoted with the same letters: what really matters is that all the constants can be taken independent of $\varepsilon$, $\rho$, $k$ and $\nu$. Notice also that the proof will be the same for the two functions $j_{0,k}=\vert \psi_k\vert$ and $j_{1,k}=\vert \psi'_k-(\gamma-1)\rho^{-1}\psi_k \vert$ in \eqref{moduleigen}, since the corresponding two versions of formula \eqref{moduleigen} only differ by the harmless factor $t^\varepsilon$ in the integrand.
Another point is that the value of $j_{\varepsilon,k}(\rho)$ remains unchanged if we simultaneously replace $\nu$ by $-\nu$ and $\rho$ by $-\rho$ (this just acts by complex conjugation on $\Gamma(\gamma-i\nu)$ and the integrand of $I_{\varepsilon,\gamma,\rho}$). So, in the sequel of this proof, without restricting the generality we shall only consider positive values of $\rho$, but we make no sign assumption on $\nu$.

In the right-hand side of formula \eqref{moduleigen}, the integral is multiplied by the prefactor $\frac{e^{\frac{\pi\nu}2}\rho^{\gamma-1}
}{2^{\gamma+1/2}|\Gamma(\gamma-i\nu)|}$ that has to be estimated.
Stirling's formula
$\lim_{|z|\rightarrow +\infty}\frac{\sqrt{z}\Gamma(z)}{\left(\frac{z}e\right)^z\sqrt{2\pi}}=1
$
implies that 
$\frac{1}{\Gamma(\gamma-i\nu)}=\mathcal O\left((\frac{e}{\vert k\vert})^{\gamma-1/2}\right)$ for $\gamma\gg1$. On the other hand, the function $1/\Gamma(z)$ is bounded on the set $\{z\in\R_++i[-1,1]\,:\vert z\vert\geq 1\}$, and for any $k\in\Z^*,\,\nu\in [-1,1]$, the numbers $\gamma-i\nu$ and $\gamma+1$ belong to this set. As a consequence we get, for any $k\in \Z^*$ and $\nu\in [-1,1]\,,$ the two estimates

\begin{equation}\label{prefactor}
\frac{e^{\frac{\pi\nu}2}\rho^{\gamma-1}
}{2^{\gamma+1/2}|\Gamma(\gamma-i\nu)|}\leq C\left(\frac{ \:e}{2|k|}\right)^{\gamma-1/2}\rho^{\gamma-1}\,,
\end{equation}
\begin{equation}\label{prefactorbis}
\frac{e^{\frac{\pi\nu}2}\rho^{\gamma-1}
}{2^{\gamma+1/2}|\Gamma(\gamma-i\nu)|}\leq C\frac{(\gamma+1)\rho^{\gamma-1}}{2^{\gamma}\Gamma(\gamma+1)}\,.
\end{equation}
We now have to bound the integral $I_{\varepsilon,\gamma,\rho}\,$ given by formula \eqref{integral} in order to prove estimate \eqref{esteigformula}. We recall our notations $\,q=\frac{\gamma-1}{\rho}\,,$ $\,g_\varepsilon(z)=z^\varepsilon(1+z)^{-i\nu}(1-z)^{1+i\nu}$ and $\,h_q(z)=-iz+q\ln(1-z^2)$ for $z$ in $\Omega=\C\setminus ((-\infty,-1]\cup [1,\infty))\,.$\medskip

\noindent
 For technical reasons, we split our set of parameters into three sectors, and in the third one we distinguish several cases and subcases depending on the value of $q$. 

\subsection{The sector $0<\rho\leq\max\{\vert k\vert/2,2\}$}\label{sector 1} In this region we do not need the method of steepest descent.\medskip

First of all, $I_{\varepsilon,\gamma,\rho}\,$ is uniformly bounded when $\gamma\geq 1$. Indeed, its integrand $\phi(t)=e^{-i\rho t}t^\varepsilon(1+t)^{\gamma-1-i\nu}(1-t)^{\gamma+i\nu}$ has modulus
$\,t^\varepsilon(1-t^2)^{\gamma-1}(1+t)\leq (1+t)\,.$\medskip

When $\gamma\in [0,1)$, which means that $\vert k\vert=\vert\gamma-i\nu\vert=1$, we may write $I_{\varepsilon,\gamma,\rho}=I^-+I^+$ with
$I^-=\int_{-1}^0\phi(t) dt\ ,\ 
I^+=\int_{0}^1\phi(t) dt\,.$ For $t\in (0,1)$ we have $\vert\phi(t)\vert\leq 1$, hence $\vert I^+\vert\leq 1$. For $t\in (-1,0)$ we write $\phi(t)=u(t)v'(t)$ with $u(t)=e^{-i\rho t}t^\varepsilon(1-t)^{\gamma+i\nu}$ and $v(t)=(\gamma-i\nu)^{-1}(1+t)^{\gamma-i\nu}$, hence, after integration by parts,
$\vert I^-\vert\leq C(1 + \rho)\,.$\medskip

Gathering the above estimates on $I_{\varepsilon,\gamma,\rho}$ and combining them with \eqref{prefactor}, we get the bounds
\begin{equation}\label{est1}
j_{\varepsilon,\pm 1}(\rho)\leq C\rho^{\gamma-1}(1+\rho)\ \hbox{ and }\ j_{\varepsilon,k}(\rho)\leq C\left(\frac{ \:e}{2| k |}\right)^{\gamma-1/2}\rho^{\gamma-1}\ \hbox{ for }\ \vert k\vert\geq 2\,\,.
\end{equation}

Using \eqref{est1}, we easily get an estimate of the form
$
j_{\varepsilon,k}(\rho)\leq C e^{-D\vert k\vert}(\rho/2)^{\gamma-1}\,
$
in the region $0<\rho\leq 2\,,$ $k\in \Z^*\,$. In the region $2\leq\rho\leq\frac{\vert k\vert}{2}$, using \ref{est1} again, we see that 
$j_{\varepsilon,k}(\rho)\leq C \left(\frac{e}4\right)^\gamma\leq C e^{-D\vert k\vert}\,.
$
The combination of these two estimates gives \eqref{esteigformula} in the region $0<\vert \rho \vert \leq \max\{\vert k \vert/2\,,\,2\}\,,$ $k\in \Z^*\,$.\medskip

\subsection{The sector $\rho \geq\max\{2\,, (\gamma+1)^2/2\}$}\label{straight lines}

In this region, $q$ tends to zero sufficiently fast when $\rho$ goes to infinity and it is sufficient to use a contour made of the steepest descent curves of $\Re\{h_0(z)\}=-\Im\{z\}$. These curves are just straight lines parallel to the imaginary axis, so we are going to integrate the holomorphic function $g_\varepsilon e^{\rho h_q}$ on $\Gamma_0=(-1,-1-i\infty)\cup (1-i\infty,1)\,.$ It is easy to justify that the integral on $\Gamma_0$ coincides with $I_{\varepsilon,\gamma,\rho}$, by, first, deforming $(-1,1)$ into the bounded contour $(-1,-1-iA]\cup [-1-iA,1-iA]\cup [1-iA,1)$ for $A>0$, then passing to the limit $A\to\infty$, thanks to the exponential decay of $g_\varepsilon(z) e^{\rho h_q(z)}$ when $\Im(z)\to-\infty\,$.

\begin{figure}[htp]
\begin{center}
\includegraphics[width=14cm,height=12cm]{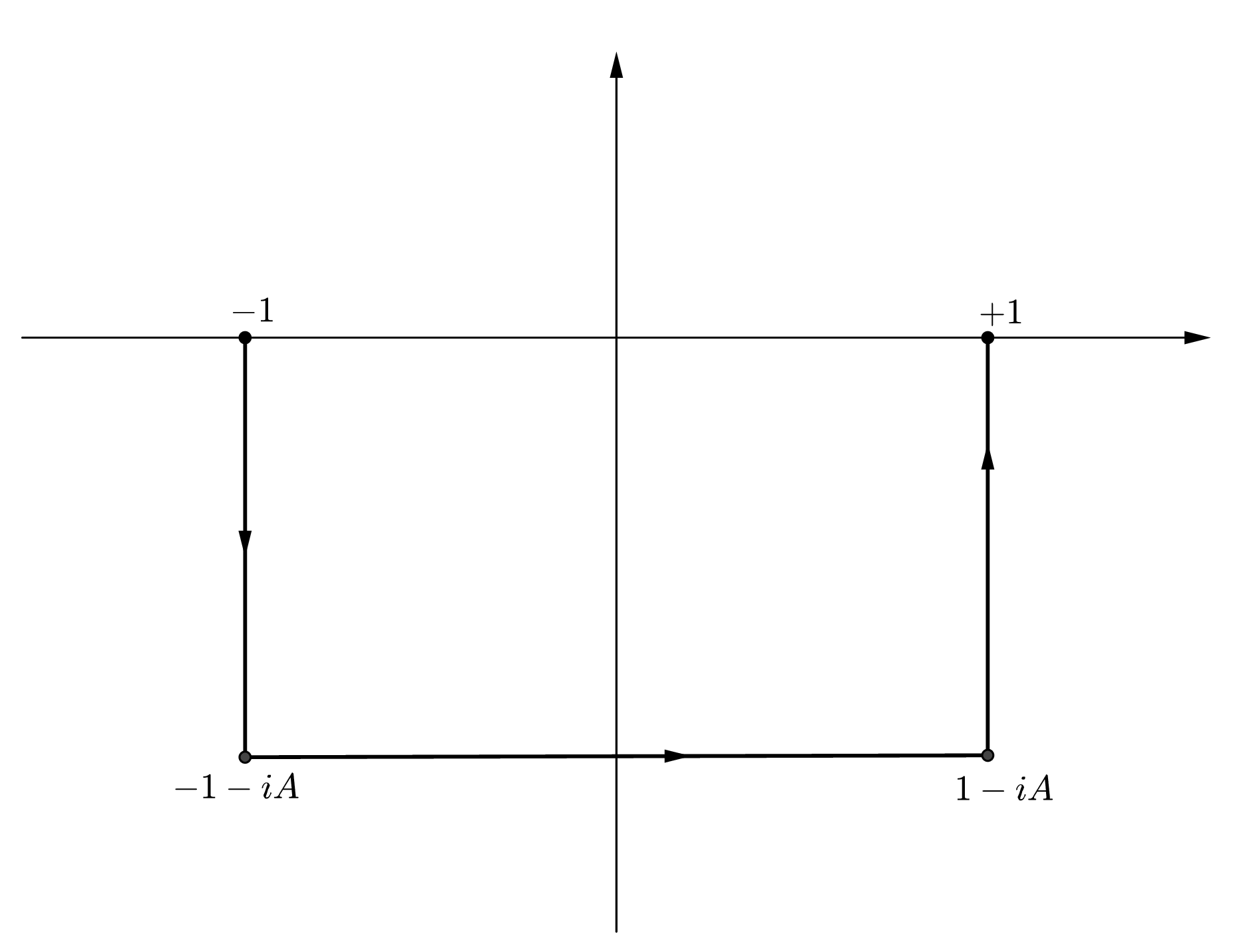}
\hspace{3mm}
\caption{\small{The deformed path.}}
\end{center}
\end{figure}

So we may write $I_{\varepsilon,\gamma,\rho}=I^{(-1)}+I^{(1)}$ with
$$I^{(-1)}=\int_{-1}^{-1-i\infty} g_\varepsilon(z) e^{\rho h_q(z)}dz=-i\int_{0}^\infty e^{-\rho t}(-1-it)^\varepsilon(-it)^{\gamma-1-i\nu}(2+it)^{\gamma+i\nu} dt\,,$$
$$I^{(1)}=\int_{1-i\infty}^{1} g_\varepsilon(z) e^{\rho h_q(z)}dz=i\int_{0}^\infty e^{-\rho t}(1-it)^\varepsilon (it)^{\gamma+i\nu}(2-it)^{\gamma-1-i\nu} dt\,.$$
Then, for any $\gamma\geq 0$,
$$\vert I^{(1)}\vert\leq 2^{\gamma}\int_{0}^{\infty} e^{-\rho t}\,t^{\gamma}\left(1+\frac{t^2}{4}\right)^{\frac{\gamma}{2}} dt\,.$$
Since 
$
(1+a^2)^b\leq (1+a)^{2b}\leq e^{2ab}\,
$ for all $a,\,b\in\R_+\,$, when $0\leq \gamma< 2\rho\,$ the above inequality implies
\begin{equation}\label{I(1)}
\vert I^{(1)}\vert\leq C\, 2^{\gamma} \int_{0}^{\infty} e^{-\left(\rho-\frac\gamma2\right) t}\,t^{\gamma} dt=C\, 2^{\gamma}\left(\rho-\frac\gamma2\right)^{-\gamma-1}\Gamma(\gamma+1)\,.
\end{equation}
Similarly, when $1\leq\gamma < 2\rho-1\,$, we get
$I^{(-1)}\leq C\, 2^{\gamma+1}\left(\rho-\frac{\gamma+1}{2}\right)^{-\gamma}\Gamma(\gamma)\,.$ Combining these two estimates and using \eqref{moduleigen} together with \eqref{prefactorbis}, we get
\begin{equation}\label{decay1}
j_{\varepsilon,k}(\rho)\leq \frac{C}{\rho}\left(1-\frac{\gamma+1}{2\rho}\right)^{-\gamma-1}\leq \frac{C'}{\rho}\ \hbox{ for } \ \rho\geq (\gamma+1)^2/2\,,\ \vert k\vert\geq 2\,.
\end{equation}
For $\gamma\in [0,1)$ {\it i.e.} $\vert k\vert =\vert \gamma-i\nu\vert=1$, an integration by parts gives us
$$I^{(-1)}=-\int_{0}^\infty \frac{(-it)^{\gamma-i\nu}}{(\gamma-i\nu)}\frac{d}{dt}\left( e^{-\rho t}(-1-it)^\varepsilon(2+it)^{\gamma+i\nu}\right) dt\,,$$
so, after a short calculation, for $\rho\geq 2$ we find
$$
\vert I^{(-1)}\vert\leq C(\rho+1)\int_{0}^{\infty} e^{-\rho t}\,t^{\gamma}\left(1+\frac{t^2}{4}\right)^{\frac{\gamma+1}{2}} dt
\leq
C(\rho+1)\left(\rho-\frac{\gamma+1}{2}\right)^{-\gamma-1}\Gamma(\gamma+1)\,.
$$
Combining this inequality with \eqref{prefactorbis} and \eqref{I(1)} we get the bound $j_{\varepsilon,k}(\rho)\leq \frac{C''}{\rho}$ for $k=\pm 1\,,\ \rho\geq 2\,.$ This, together with \eqref{decay1}, ends the proof of \eqref{esteigformula} in the sector $\rho \geq\max\{2\,, (\gamma+1)^2/2\}$, for all $k\in\Z^*$ and $\nu\in [-1,1]\,.$\medskip

\subsection{The remaining sector: $\vert k\vert\geq 2\,$ and $\vert k\vert/2\leq \rho \leq (\gamma+1)^2/2$}

Here we have to use the steepest descent curves of $\Re(h_q)$ for $q=\frac{\gamma-1}{\rho}\in (0,2)\,.$\medskip

To find the saddle-points of $h_q$, we compute
\begin{equation}\label{h'eq}
h_q'(z)=-i+\frac{q}{z-1}+\frac{q}{z+1},
\end{equation}
so 
\begin{equation}\label{second degree}
h_q'(z)=0 \quad\Leftrightarrow\quad z^2+2iqz-1=0.
\end{equation}

Given that the discriminant of the equation above is $4(1-q^{2})$, we will treat separately the cases $0< q<1$ and $1\leq q\leq 2$, introducing some additional subcases in order to help our presentation. The thresholds $q_0,\,q_1,\,q_2$ that separate the subcases satisfy the conditions $2>q_0>1>q_1>q_2>0$. They will be chosen during the proof: $q_0-1\,,\,1-q_1$ and $q_2$ will have to be sufficiently small.
\medskip

\begin{enumerate}
\item {\bf Case $1\leq q\leq 2$}. The saddle-points are 
$
z_\pm^0=-i\left(q\pm\sqrt{{q^2}-1}\right)\,.
$
They coalesce when $q=1\,.$
We will see that the path $\Gamma_q$ connects $-1$ to $1$ by passing through $z_-^0$.\medskip

\noindent
Subcases: 
\begin{enumerate}
\item $q_0\leq q\leq 2\,$: $\vert z_+^0-z_-^0\vert$ is bounded away from zero.
\item $1\leq q\leq q_0\,$: $\vert z_+^0-z_-^0\vert$ is very small and vanishes when $q=1$.
\end{enumerate}
\medskip

\noindent
\item {\bf Case $0< q<1$}. The saddle-points are
$z^0_\pm=\pm\sqrt{1-q^2}-iq
\,.$
We will see that $\Gamma_q$ has two unbounded branches: one starting at $-1$ and passing through $z^0_-$, the other passing through $z^0_+$ and ending at $1$.\medskip

\noindent
Subcases:
\begin{enumerate}
\item
$\displaystyle q_2\leq q\leq q_1\,$: $\vert z_+^0-z_-^0\vert$, $\vert z_-^0+1\vert$ and $\vert z_+^0-1\vert$ are bounded away from zero.
\medskip
\item 
$\displaystyle q_1\leq q<1\,$: $\vert z_+^0-z_-^0\vert$ is very small.
\medskip
\item
$0< q\leq q_2\,$: $\vert z_-^0+1\vert$ and $\vert z_+^0-1\vert$ are very small.
\end{enumerate}
\end{enumerate}

\medskip

Before getting to the details of each case, let us briefly comment on the overall strategy. In each of the two macro-cases {\em (1)-(2)}, we will build a parametrized path $\Gamma_q$ of ``steepest descent" for $\Re\{h_q\}$ starting from $-1$ and ending at $1$. It will turn out that $\Gamma_q$ passes through one saddle point in case {\em (1)} and through both of them in case {\em (2)}. Once this construction is made, we will Taylor expand $\Re\{h_q(z)\}$ around the saddle points; this, together with some uniform estimates on $|g_\varepsilon(z)|$, allows the application of Laplace's method to estimate the integral $I_{\varepsilon,\gamma,\rho}$. Notice that this strategy works ``effortlessly" when $q$ stays away from $1$ and $0$, i.e. subcases {\em (1.a)} and {\em (2.a)}. Indeed, in those subcases the saddle points stay away from each other and from the branch points $\pm 1$, so the standard Laplace method applies and gives estimates that are uniform in $q$. The subcases {\em (1.b)} and {\em (2.b)} will need additional care: they correspond to the zone $q\sim1$ of coalescence of the two saddle-points. In this zone one has to go to third order in the Taylor expansion of $\Re\{h_q\}$, moreover the behaviour of $\Gamma_q$ near the saddle-points is more complicated. For this reason we will locally replace $\Gamma_q$ by a carefully chosen piecewise affine curve: this will simplify the computations. The subcase {\em (2.c)} presents another type of difficulty: when $q\rightarrow 0$, the saddle points $z^0_\pm$ converge to $\pm 1$. To overcome this last problem, we will need to locally rescale the parametrization of $\Gamma_q$ and to perform some double expansions in powers of $q$ and of the rescaled parameter $u$.

\medskip

We now start dealing with the cases and subcases one by one. Thanks to the identity $h_q(-\overline{z})=\overline{h_q(z)}$, our steepest descent paths $\Gamma_q$ will be symmetric with respect to the imaginary axis. Moreover it will turn out that $\Gamma_q$ lies in the domain $\mathcal D =(-1,1)+i(-\infty,0)$. In this domain we have $\Re\{1-z^2\}>0$, so the real and imaginary parts of $h_q(z)=-iz+q\ln(1-z^2)$ are given by
\begin{equation}\label{Re h}
\Re\{h_q(z)\}= \Im\{z\}+q\ln\vert1-z^2\vert=y+\frac{q}{2}\ln(\,(1-x^2+y^2)^2+4x^2y^2\,)\,,
\end{equation}

\begin{equation}\label{Im h}
\Im\{h_q(z)\}=-\Re\{z\}+q\arg (1-z^2)=-x-q\arctan\left(\frac{2xy}{1-x^2+y^2}\right)\,
\end{equation}
for any complex number $z=x+iy$.

\medskip

\subsubsection{Case (1): $1 < q\leq 2$ } Here, as already mentioned, the solutions to \eqref{h'eq} are purely imaginary:
\begin{equation}
z_\pm^0=-i\left(q\pm\sqrt{{q^2}-1}\right)
\end{equation}
It is not difficult to check that $\Re\{h_q(z^0_-)\}<\Re\{h_q(z^0_+)\}$ and $\Im\{h_q(z^0_\pm)\}=0$. Moreover, any complex number $z=x+iy\in \mathcal D$ satisfies $x/q\in (-\pi/2,\pi/2)$, so, from \eqref{Im h}, one has $\Im\{h_q(z)\}=0$ if and only if $x=0$ or $\,y^2+2\frac{x}{\tan(x/q)} y+(1-x^2)=0\,.$ So $h_q^{-1}(0)\cap \mathcal D=(-i\infty,0)\cup\Gamma_-\cup\Gamma_+\,$, where
$\Gamma_\pm=\left\{t+iy_\pm(t)\,,\;t\in(-1,1)\,\right\}$ and
\begin{equation}\label{curv}
\begin{cases}
y_\pm(t)=-\frac{t}{\tan(t/q)}\mp\sqrt{\frac{t^2}{\sin^2(t/q)}-1}\quad t\in (-1,0)\cup(0,1);\\
y_\pm(0)=-q\mp\sqrt{{q^2}-1}.
\end{cases}
\end{equation}
We have $\,\lim_{t\to\pm 1} y_-(t)=\pm 1\,,$ so $\pm 1$ are the end points of $\Gamma_-$ and it is natural to choose $\Gamma_q=\Gamma_-\,,$ {\it i.e.} $z_q(t)=t+iy_-(t)\,,$ $t\in (-1,1)\,.$ The maximum of $\Re\{h_q\circ z_q\}$ is attained at $z_q(0)=z^0_-\,.$

\begin{figure}[htp]
\begin{center}
\includegraphics[width=10cm,height=8cm]{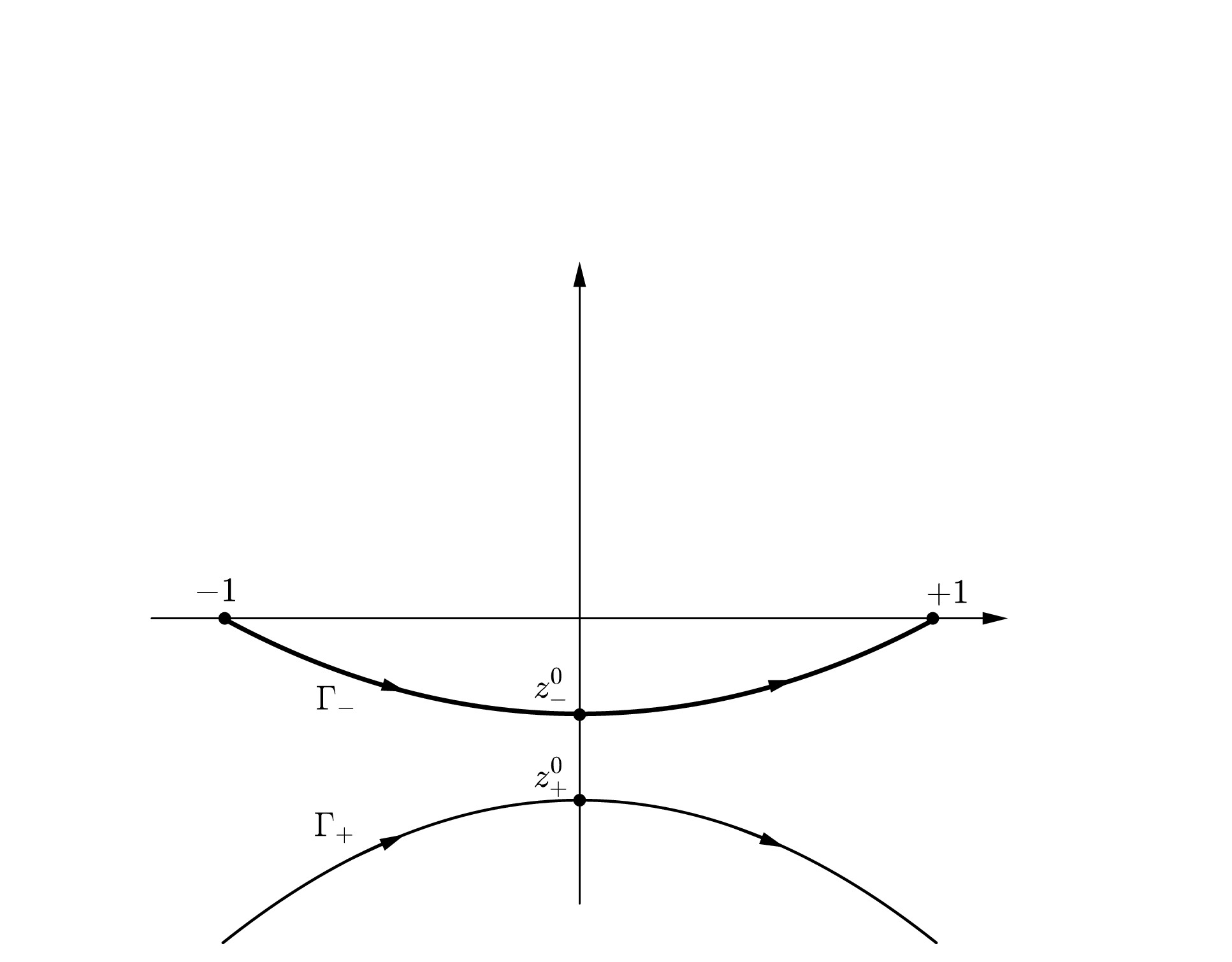}
\hspace{3mm}
\caption{\small{Picture of the paths $\Gamma_\pm$ defined by \eqref{curv}.}}
\label{fig1}
\end{center}
\end{figure}

 Note that the above arguments and formulas are still valid when $q=1\,,$ the only difference in this limiting case being that $\Gamma_-\,,\,\Gamma_+$ intersect at $z_-^0=z_+^0=-i$ and are not differentiable at this point.\medskip
 
We now need to distinguish further between the subcases when $q$ is significantly larger than $1$ and $q$ is close to $1$.

\medskip

{\em Subcase (1.a): $1<q_0\leq q \leq 2$.}
The estimate given in this subcase depends on the threshold $q_0>1$, that will be chosen in  {\em (1.b)}. We have

\begin{equation}\label{reder}
h_q(z^0_-)=q\ln\left(\frac{2q}{e}\right)+\sqrt{q^{2}-1}-q\ln\left(q+\sqrt{q^2-1}\right)\,.
\end{equation}

Note that the function ${\ell(q)}=\sqrt{q^{2}-1}-q\ln\left(q+\sqrt{q^2-1}\right)$ is decreasing on $[1,\infty)$. Indeed, if we make the change of variables $q=\cosh r\,,\,r\geq 0$, ${\ell(q)}$ becomes $\sinh(r)-r\cosh(r)$ whose derivative $-r\sinh(r)$ is negative. So, for $q\geq q_0>1$, we have ${\ell(q)}\leq {\ell(q_0)}<0\,.$ Since the maximum of $\Re\{h_q(z)\}$ on $\Gamma_q$ is attained at $z=z_-^0$, we conclude from \eqref{reder} that
$\vert e^{\rho h_q(z)}\vert \leq \left(\frac{2q}{e}\right)^{\gamma-1} e^{{\ell(q_0)}\rho}\,,\ \forall z\in \Gamma_q\,.$\medskip

Now, for $q$ in the compact set $[1,2]$, the length of the path $\Gamma_q$ is uniformly bounded as well as $\sup_{z\in\Gamma_q} \vert g_\varepsilon(z)\vert\,.$ As a consequence,
$$\vert I_{\varepsilon,\gamma,\rho}\vert=\left\vert\int_{\Gamma_q} g_\varepsilon(z)e^{\rho h_q(z)}dz\right\vert\leq C\left(\frac{2q}{e}\right)^{\gamma-1}e^{{\ell(q_0)}\rho}\,.$$
Combining this with \eqref{moduleigen}, \eqref{prefactor}, we get
\begin{equation}\label{estk1}
j_{\varepsilon,\gamma}(\rho)\leq C\vert k\vert^{-1/2}\,e^{{\ell(q_0)}\rho}\ \hbox{ for }\ q_0\rho+1\leq \gamma \leq 2\rho +1\,.
\end{equation}

Note that we could have obtained a more precise asymptotic estimate for $\rho$ large by using the full power of Laplace's method, but \eqref{estk1} is already much stronger than \eqref{esteigformula} in the domain $\vert k\vert/2\leq\rho\leq (\gamma-1)/q_0\,.$ It even allows to increase the domain of validity of the estimate $j_{\varepsilon,\gamma}(\rho)\leq C(\min\{\rho/2\,,\,1\})^{\gamma-1}\,e^{-D\vert k\vert}\,$. This exponential bound is stated in \eqref{esteigformula} and proved in subsection \ref{sector 1} for the sector $0< \vert \rho\vert \leq \vert k\vert/2\,$. Using \eqref{estk1}, one can extend it to the larger sector $0<\vert\rho\vert \leq (1-\eta)\vert k\vert$ for any $0<\eta<1/2\,,$ the positive constant $D$ depending only on $\eta$. But $D(\eta)$ tends to zero  as $\eta\to 0\,$, so in order to estimate properly $j_{\varepsilon,\gamma}(\rho)\,$ when $q$ approaches $1$, more work is needed. This leads us to the next subcase.

\medskip

{\em Subcase (1.b): $1\leq q\leq q_0\,$.}
This subcase is delicate, due to the degeneracy at $q=1$.
We need precise informations on the behaviour of the derivatives of $h_q\,$ at $z^0_-\,$ when $q-1$ becomes small. First of all, in our study of subcase $(1.a)\,$ we have proved that
\begin{equation}
\label{reder0}
h_q(z^0_-)\leq q\,\ln\left(\frac{2q}{e}\right)\,,\quad \forall q\geq 1\,.
\end{equation}
We recall in addition that $h'_q(z^0_-)=0$. For the higher derivatives, we first compute
$h_q''(-i)=0\,,$ $h_q^{(3)}(-i)=qi$ and we see that for some $M>0$, the fourth derivative $h^{(4)}_q(z)=-\frac{6q}{(z-1)^4}-\frac{6q}{(z+1)^4}$ satisfies
\begin{equation}
\label{reder4}
\vert h_q^{(4)}(z)\vert\leq M\,,\ \forall z\in [-1/2,1/2]+i\R\,,\;\forall q\in [1,2]\,.
\end{equation}
Moreover
$
z^0_-=-i+i\sqrt{2(q-1)}+O(q-1)\,,
$
hence
\begin{equation}
\label{reder2}
h''_q(z^0_-)=-\sqrt{2(q-1)}+O(q-1)\,,
\end{equation}
\begin{equation}\label{reder3}
h^{(3)}_q(z^0_-)=i+O(\sqrt{q-1})\,.
\end{equation}

For $q>1$ the function $z_q$ is of class $C^1$, with $z_q(0)=z^0_-$ and $z'_q(0)=1\,$. But when $q\to 1\,,$ $z_q$ converges uniformly to $z_1\,$ on $(-1,1)$ and we have, for $\vert t\vert\ll 1\,,$
\begin{equation}\label{development}
z_1(t)=-i+t+i\vert t\vert/\sqrt{3}+o(t)=
\begin{cases}
-i+\frac{2}{\sqrt{3}}e^{i\pi/6}t+o(t)\;\,\quad\hbox{ if }\, t\geq 0\,,\\
-i+\frac{2}{\sqrt{3}}e^{5i\pi/6}t+o(t)\,\quad\hbox{ if }\, t\leq 0\,.

\end{cases}
\end{equation}

For $q\geq 1$ and $\delta\in (0,1/2)\,$, let $\alpha_q^\pm=\arg\{z_q(\pm\delta)-z_q(0)\}\,$. Since $\Re\{z_q\}(t)=t$ and $\Im\{z_q\}$ is even, we have $\alpha_q^-=\pi-\alpha_q^+\,$ and $z_q(\pm\delta)=\delta(\pm 1 + i\tan \alpha_q^+)=\frac{\delta }{\cos \alpha_q^+}e^{i\alpha^\pm_q}$.
From \eqref{development}, choosing $\delta$ small enough we may impose $\vert\alpha^+_1-\pi/6\vert<\pi/24$, hence $\vert\alpha^-_1-5\pi/6\vert<\pi/24\,.$ Then $\alpha_q^\pm$ will satisfy the same estimates for $q-1>0$ small enough. So, using \eqref{reder2}\eqref{reder3}, we may choose a threshold $q_0>1$ (that depends on $\delta$) such that, if $1\leq q\leq q_0$
then
$$\Re\{h''_q(z^0_-)e^{2i\alpha^\pm_q}\}\leq -\kappa\sqrt{q-1}\ \hbox{ and } \ \Re\{h^{(3)}_q(z^0_-)e^{3i\alpha^\pm_q}\}\leq -\kappa\,,$$
for some positive constant $\kappa\,$ independent of $\delta$ and $q\in [1,q_0]$. Then, using \eqref{reder4} and taking a smaller $\delta$ (hence a smaller $q_0$) if necessary, by Taylor's formula we can get an estimate of the form
\begin{equation}\label{descent}
\Re\{h'_q(z^0_-+s\,e^{i\alpha^\pm_q})e^{i\alpha^\pm_q}\}\leq -\kappa \left(\sqrt{q-1}\,s+s^2/4\right)\,,\ \forall s\in \left[0\,,\,\frac{\delta}{\cos \alpha_q^+}\right].
\end{equation}

This inequality gives us a controlled rate of descent for $\Re\{h_q\}$ along the segments $[z^0_-,z_q(\pm\delta)]$.
The idea now is to replace the steepest descent path $\Gamma_q$ by these ``segments of controlled descent" near $z^0_-$. The new integration curve is thus
\begin{equation}\label{Gammatilde}
\tilde\Gamma_q:=\tilde\Gamma_q^{(-2)}\cup\tilde\Gamma_q^{(-1)}\cup\tilde\Gamma_q^{(1)}\cup\tilde\Gamma_q^{(2)}
\end{equation}
where:
\begin{eqnarray*}
&\tilde\Gamma_q^{(-2)}&=\left\{z_q(t)\, :\ t\in[-1,-\delta]\right\};\\
&\tilde\Gamma_q^{(-1)}&=\big[z_q(-\delta), z_q(0)\big];\\
&\tilde\Gamma_q^{(1)}&=\big[z_q(0), z_q(\delta)\big];\\
&\tilde\Gamma_q^{(2)}&=\left\{z_q(t)\, :\ t\in[\delta,1]\right\}\,.
\end{eqnarray*}

\begin{figure}[htp]
\begin{center}
\includegraphics[width=15cm,height=7cm]{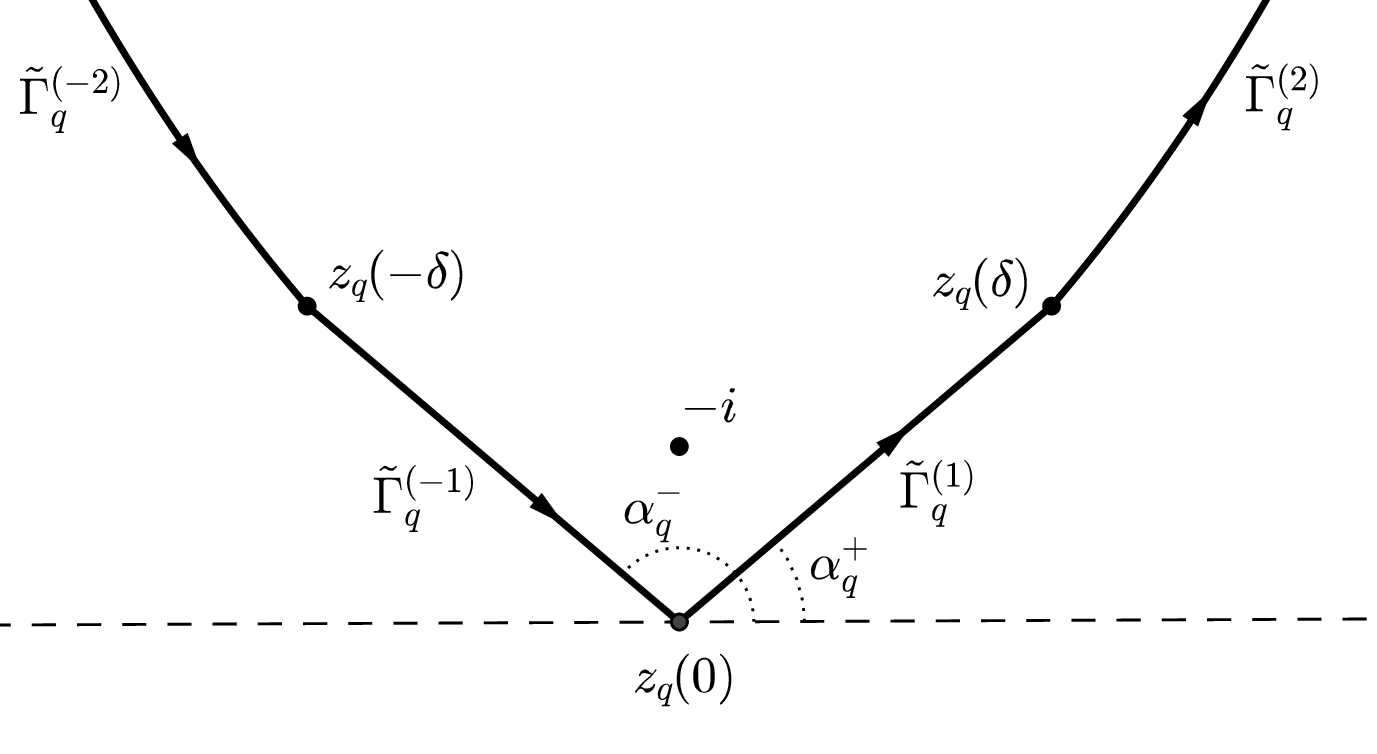}
\hspace{3mm}
\caption{\small{Picture of the curve $\tilde\Gamma_q$ as defined in \eqref{Gammatilde}.}}
\label{fig2}
\end{center}
\end{figure}
Inequalities \eqref{reder0}\eqref{descent} give an estimate on $\Re\left\{h_q(z)\right\}$ for $z$ in $\tilde{\Gamma}_q^{(-1)}\cup\tilde\Gamma_q^{(1)}$: there is $\kappa'>0\,$ such that for all $q\in [1,q_0]$ and $t\in [-\delta,\delta]\,$,
\begin{equation}\label{boundpm1}
\Re\{h_q(z^0_-+t + i\vert t\vert\tan \alpha_q^+)\}\leq q\,\ln\left(\frac{2q}{e}\right)-\kappa' \left(\sqrt{q-1}\,t^2+\vert t\vert^3\right)\,.
\end{equation}
On $\Gamma_q^{(\pm 2)}$,
$\Re\left\{h_q(z)\right\}$ achieves its maximum at $z=z_q(\pm\delta)$, so using formula \eqref{boundpm1} for $t= \delta$ we get the uniform bound
\begin{equation}\label{boundpm2}
\sup\left\{ \Re\left\{h_q(z)\right\}\,:\;z\in \tilde{\Gamma}_q^{(-2)}\cup\tilde\Gamma_q^{(2)}\right\} \leq q\,\ln\left(\frac{2q}{e}\right)-\kappa' \,\delta^3\,.
\end{equation}

Recalling that $\sup_{\tilde{\Gamma}_q}\vert g_\varepsilon\vert$ and the length of $\tilde{\Gamma}_q$ are uniformly bounded when the parameter $q$ varies in $[1,q_0]\,$, we now estimate $j_{\varepsilon,k}(\rho)= \frac{e^{\frac{\pi\nu}2}\rho^{\gamma-1}
}{2^{\gamma+1/2}|\Gamma(\gamma-i\nu)|}
|I_{\varepsilon,\gamma,\rho}|$ with $I_{\varepsilon,\gamma,\rho}=\int_{\tilde{\Gamma}_q} g_\varepsilon(z)e^{\rho h_q(z)}dz\,$.\medskip

\noindent
From \eqref{boundpm1},
$$ \left\vert\int_{\tilde{\Gamma}_q^{(-1)}\cup\tilde{\Gamma}_q^{(1)}} g_\varepsilon(z)e^{\rho h_q(z)}dz\right\vert \leq C\left(\frac{2q}{e}\right)^{\gamma-1}\int_\R e^{-\rho\kappa' (\sqrt{q-1}\,t^2+\vert t\vert^3)}dt$$
and from \eqref{boundpm2},
$$ \left\vert\int_{\tilde{\Gamma}_q^{(-2)}\cup\tilde{\Gamma}_q^{(2)}} g_\varepsilon(z)e^{\rho h_q(z)}dz\right\vert \leq C\left(\frac{2q}{e}\right)^{\gamma-1} e^{-\rho\,\kappa' \delta^3}\,.$$
So, using \eqref{prefactor}, we find the estimate

\begin{eqnarray*}
j_{\varepsilon,k}(\rho)&\leq& C \vert k\vert^{-1/2}\left( \int_\R e^{-\rho\kappa' (\sqrt{q-1}\,t^2+|t|^3)}dt +  e^{-\rho\,\kappa' \delta^3}  \right)\\
&\leq& C \vert k\vert^{-1/2} \left(\min\left\{\int_\R e^{-\rho\,\kappa' \sqrt{q-1}\,t^2}dt,\int_\R e^{-\rho\kappa' |t|^3}dt\right\} +e^{-\rho\,\kappa' \delta^3}\right)
\\
&\leq &
C' \vert k\vert^{-1/2} \min\left\{ (q-1)^{-\frac14}\rho^{-\frac12}, \rho^{-\frac13}\right\}
\\
&\leq &
C'' \vert k\vert^{-\frac34}\left(\vert k\vert^{\frac13}+|\vert k\vert-\rho|\right)^{-\frac14}.
\end{eqnarray*}

This ends the proof of \eqref{esteigformula} for $1\leq q\leq 2\,$. Note that for $q=1$ our estimate becomes $j_{\varepsilon,k}(\vert k \vert)=0(\vert k\vert^{-\frac56})$ and the exponent $-\frac56$ is optimal, as one easily checks by applying the standard asymptotic Laplace method to the integral $\int_{\Gamma_1} g_\varepsilon(z)e^{\rho h_1(z)}dz\,.$

\medskip

\subsubsection{ Case 2: $0< q<1$.} Introducing the angle $\theta_0=\arccos q\,$,
we may write the solutions of \eqref{second degree} as 
\begin{equation}\label{sadq>1}
z^0_\pm=e^{ i(-1\pm\theta_0)}\,.
\end{equation}

Easy calculations yield:
\begin{eqnarray}\label{derq>1}
\nonumber
&\Re\{h_q(z^0_\pm)\}&=q\,\ln\left(\frac{2q}{e}\right)\,;
\\
&h''_q(z^0_\pm)&=\tan \theta_0\, e^{\,\pm i(\frac{\pi}{2}-\theta_0)}\,;
\\
\nonumber
&h^{(3)}(z^0_\pm)&=e^{i\frac{\pi}{2}}+O(\theta_0)\,;
\\
\nonumber
&h^{(4)}(z)&=O(\vert z-1\vert^{-4}+\vert z+1\vert^{-4})\,.
\end{eqnarray}

\medskip

After some further computations we get
\begin{equation}
\Im \{h_q(z_\pm^0)\}=\pm\left(\theta_0\cos\theta_0-\sin\theta_0\right)\,.
\end{equation}

We thus have $\Im \{h_q(z_-^0)\}<0<\Im \{h_q(z_+^0)\}\,$, so we expect the path of steepest descent $\Gamma_q$ to have two components $\Gamma^l_q$ and $\Gamma^r_q\,$, with $\Gamma^l_q$ starting from $-1$ and passing through $z_-^0\,$, $\Gamma^l_q$ passing through $z_+^0\,$ and ending at $1\,$. Moreover it seems reasonable to look for
$\Gamma_q^r$ in the domain $(0,1)+i(-\infty,0)\,$ then to define $\Gamma^l_q$ as the symmetric of $\Gamma^l_q$ with respect to the imaginary axis.

\begin{figure}[htp]
\begin{center}
\includegraphics[width=13cm,height=9cm]{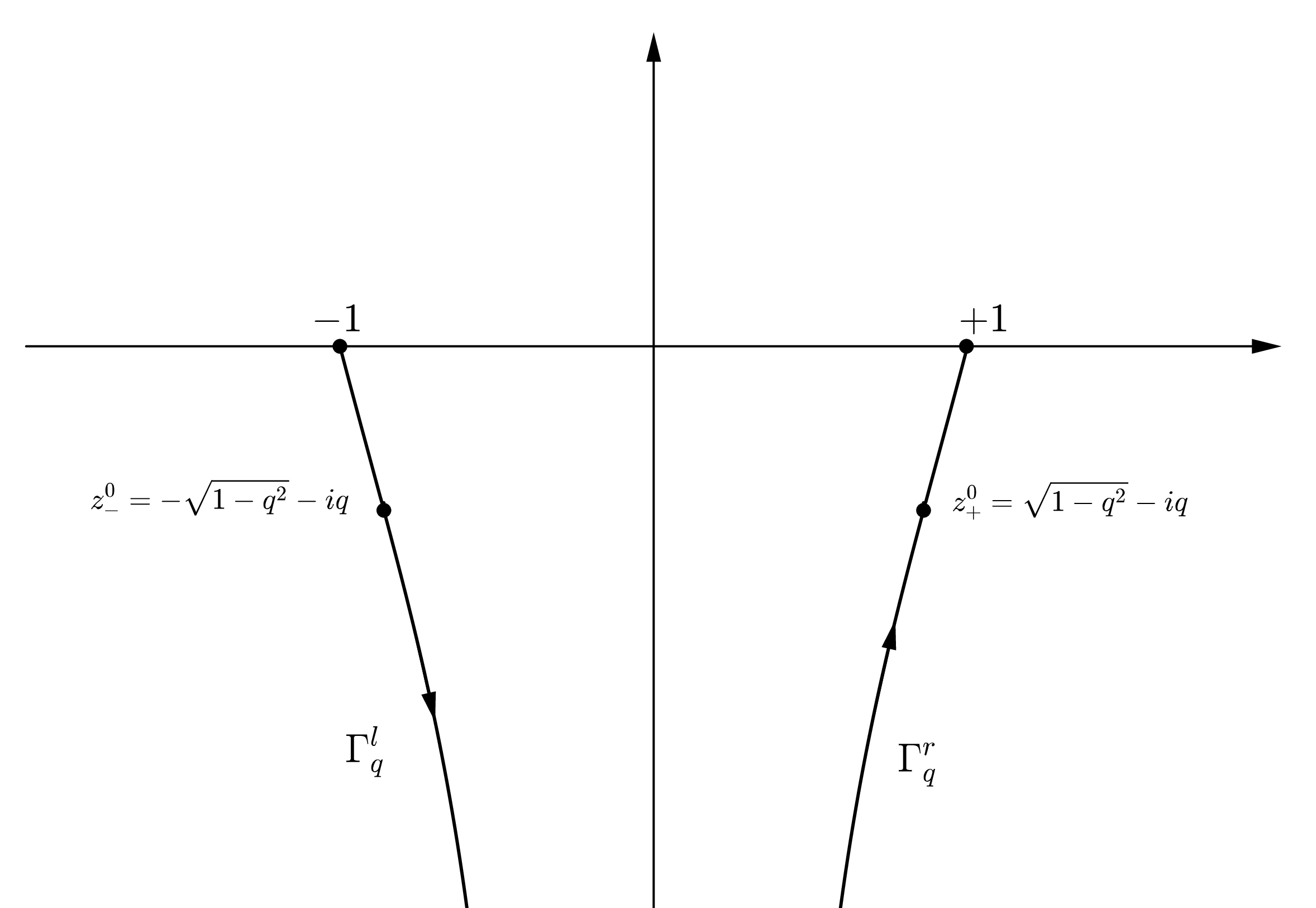}
\hspace{3mm}
\caption{\small{Picture of the curve $\Gamma_q=\Gamma^l_q\cup\Gamma^r_q$ defined above.}}
\label{fig4}
\end{center}
\end{figure}

\medskip

On $\Gamma^r_q\,$ one has $\Im \{h_q(z)\}=\Im \{h_q(z_+^0)\}$ . From \eqref{Im h}, this equation takes the form 
\begin{equation}\label{arctan}
q^{-1}(x-x_q)=-\arctan\left(\frac{2xy}{1-x^2+y^2}\right)
\end{equation}
with
\begin{equation}\label{maincod}
x_q:=\sin\theta_0-\theta_0\cos\theta_0\,\in\, (0,\sin\theta_0)\,\,.
\end{equation}
Note that $x_q$ is a decreasing function of $q\,$, $\underset{q\to 0}{\lim}\, x_q=1$ and $x_q \underset{q\to 1}{\sim} \frac{1}{3}(1-q^2)^{3/2}\,$. Since we imposed the condition $x+iy\in (0,1)+i(-\infty,0)$, if \eqref{arctan} holds then the angle $\theta:=q^{-1}(x-x_q)$ is positive, {\it i.e.} $x>x_q$. On the other hand, the condition $x<1$ is equivalent to $\theta<\theta_{\max}:=q^{-1}(1-x_q)=\frac{1-\sin\theta_0}{\cos\theta_0}+\theta_0\,$. Obviously, we have
$0<\theta_0<\theta_{\max}\,$. When $q$ decreases from $1$ to $0\,$, $\theta_0$ increases from $0$ to $\frac{\pi}2\,$ and $\theta_{\max}$ increases from $1$ to $\frac{\pi}2
\,$.\medskip

We are going to use the angle $\theta$ to parametrize $\Gamma^r_q\,$, the interval of parameters being $(0,\theta_{\max})\,$. For this purpose, we rewrite \eqref{arctan} as a system:
\begin{equation}\label{yeqfirst}
\begin{cases}
x=q\theta+x_q\,,\\
y^2+2\,\frac{q\theta+x_q}{\tan\theta}\,y+1-(q\theta+x_q)^2=0\,.
\end{cases}
\end{equation}
From the first equation, $x$ is an increasing function of $\theta$ and $\lim_{\theta\to 0} x(\theta)=x_q$, $\lim_{\theta\to \theta_{\max}} x(\theta)=1\,$. The discriminant of the second equation of unknown $y$ is $\,\Delta(q,\theta)=4\Big(\big(\frac{q\theta+x_q}{\sin\theta}\big)^2-1\Big)$. In order to study its sign, we remark that the function
$g(\theta)=q\theta+x_q-\sin\theta$ is strictly convex on $(0,\theta_{\max})$ and we check that
$g'(\theta_0)=0$. As a consequence $\Delta(q,\theta)\geq 0$ and its only zero is $\theta_0$.
\medskip

\noindent
Now, we are ready to select a family of solutions $y$ of \eqref{yeqfirst} of the form $y=-\varphi_q(\theta)$ with $\varphi_q\in C^1(0,\theta_{\max})$ and $\lim_{\theta\to \theta_{\max}}\varphi_q(\theta)=0$. These properties guarantee that $\Gamma^r_q$ is the stable 
manifold of $z_+^0$ for the flow of $\nabla \Re\{h_q\}$, since
$\lim_{z\to 1} \Re\{h_q(z)\}=-\infty.$\medskip

\noindent
We define $\varphi_q$ as follows:
\begin{equation}\label{phieq}
\varphi_q(\theta)=\frac{q\theta+x_q}{\tan\theta}-\text{sgn}(\theta-\theta_0)\sqrt{\Big(\frac{q\theta+x_q}{\sin\theta}\Big)^2-1}\,.
\end{equation}
This function is real-analytic on $(0,\theta_{\max})\,$ with $\varphi'_q(\theta_0)=\frac{1+\sin\theta_0}{\cos \theta_0}$ and satisfies $\lim_{\theta\to \theta_{\max}}\varphi_q(\theta)=0$ as demanded. Obviously, one has $\varphi_q(\theta)\geq\frac{x_q}{\tan\theta}$ for $\theta\in (0,\theta_0)$, so $\lim_{\theta\to 0}\varphi_q(\theta)=+\infty\,$.
By elementary calculations, one finds that the function $\theta\to \frac{q\theta+x_q}{\tan\theta}$ has a negative derivative on $(0,\pi/2)$ while the function $\theta\to \frac{q\theta+x_q}{\sin\theta}$ is decreasing on $(0,\theta_0)$ and increasing on $(\theta_0,\pi/2)$. This proves that $\varphi_q$ is decreasing on $(0,\theta_{\max})$. Moreover, one can get the following uniform estimate for $q\in(0,1)$ and $\theta\in (0,\theta_{\max})\,$:\medskip
\begin{equation}\label{phieq'}
\varphi_q(\theta)\,\leq\,\frac{C}{\theta}\,,\ \ \vert \varphi'_q(\theta)\vert\,\leq\,\frac{C}{\theta^2}\,,\ \ \vert \varphi''_q(\theta)\vert\,\leq\,\frac{C}{\theta^3}\,.
\end{equation}
\noindent
Finally, for $\vert\theta\vert\in (0,\theta_{\max})$ we denote
\begin{equation}\label{parametrization}
z_q(\theta)=q\theta+\text{sgn}(\theta)x_q-i\varphi_q(\vert\theta\vert)\,,
\end{equation}
and the two components of $\Gamma_q$ are
\begin{equation}\label{components}
\Gamma^l_q=\{z_q(\theta)\,:\, \theta\in (-\theta_{\max},0)\}\ , \
\Gamma^r_q=\{z_q(\theta)\,:\, \theta\in (0,\theta_{\max})\}\,.
\end{equation}
As expected, the maximum of $\Re\{h_q\}$ on $\Gamma_q$ is attained at the two saddle points $z^0_\pm=z_q(\pm\theta_0)\,$. The path $\Gamma_q$ starts from $-1$ and ends at $1\,$. It is unbounded and has $\pm x_q+(-i\infty,0)$ as asymptotes corresponding to the limits $\theta\to 0\pm\,$. One can prove the following estimate, for two positive constants $A\,,\,a$ independent of $q\in (0,1)$ and $\theta\in (0,\theta_{\max})\,$:
\begin{equation}\label{estim h}
\Re\{h_q\circ z_q(\theta)\}\leq \,A-\frac{a(1-q)^{3/2}}{\theta}\,.
\end{equation}
Moreover, from \eqref{phieq'} one can get the following bound for all $q\in (0,1)$ and $\theta\in (0,\theta_{\max})\,$:
\begin{equation}\label{estim g}
\vert (g_\varepsilon\circ z_q)(\theta) z'_q(\theta)\vert\leq C\vert\theta\vert^{-4}\,.
\end{equation}
The integral $\int_{\Gamma_q}g_\varepsilon(z)e^{\rho h_q(z)}dz$ is thus absolutely convergent. One easily checks that the value of this integral is $I_{\varepsilon,\gamma,\rho}\,,$ by first considering bounded contours, then passing to the limit, as in subsection \ref{straight lines}.

\medskip

 We now treat the subcases.

\medskip

{\em Subcase (2.a): $ q_2\leq q\leq q_1$ .}
 In this compact region the calculations are easy, as indeed here $z^0_\pm$ stay away from each other and from the singular points $\pm1\,$. Related to this nondegeneracy, one checks that $\Re\{h_q''(z^0_\pm)(z'_q(\pm\theta_0))^2\}$ is negative and bounded away from $0$, moreover $\theta_0$ stays away from the endpoints of the interval $(0,\theta_{\max})$. Using \eqref{derq>1}-\eqref{phieq}-\eqref{phieq'}-\eqref{parametrization}, one thus finds positive constants $\delta\,$, $\kappa\,$ with $\delta < \arccos q_1$ and such that, for any $q\in [q_2,q_1]$:
 
 $\bullet$ For $\vert\theta\vert \in (\theta_0-\delta,\theta_{\max})$, $\Re\{h_q\circ z_q(\theta)\}\leq q\,\ln(2q/e) - \kappa (\vert\theta\vert-\theta_0)^2\,$;
 
 $\bullet$ For $\vert\theta\vert\in (0,\theta_0-\delta)$, $\Re\{h_q\circ z_q(\theta)\}\leq \,\ln(2q/e)-\kappa\delta^2\,$.
 
 \noindent
 Then, using \eqref{estim h}-\eqref{estim g} and denoting $\mu=a(1-q_1)^{3/2}\,,$ one gets an estimate of the form
$$j_{\varepsilon,k}(\rho)\leq C \vert k\vert^{-1/2}\left( \int_{\theta_0-\delta}^{\theta_{\max}} e^{-\rho\kappa \,(\theta-\theta_0)^2}d\theta +  e^{-(\rho-1)\kappa \,\delta^2}\,\int_0^{\theta_0-\delta} \theta^{-4}e^{-\mu/\theta}d\theta  \right)\leq C'\rho^{-1}$$
for all $ q_2\leq q\leq q_1$. Of course, $C'$ depends on the thresholds $0<q_2<q_1<1$ that will be fixed in the sequel.

\medskip

{\em Subcase (2.b): $q_1\leq q<1$.}  As in subcase {\em (1.b)}, the two saddle points $z_\pm^0$ are close to each other and the second derivative of $h_q$ at these points becomes small. To derive uniform estimates, we again construct a modified path $\tilde{\Gamma}_q$, but the details of the construction differ.\medskip

When $q\to 1-\,$, one has $\theta_{\max}\to 1\,$, $x_q\sim \theta_0^3/3\,$, $\theta_0\sim \sqrt{2(1-q)}\,$, $z^0_\pm+i\sim\pm\theta_0\,$ and $h''_q(z^0_\pm)\sim \theta_0e^{\pm i\pi/2}\,$; moreover $h^{(3)}_q(z^0_\pm)$ converges to $e^{i\pi/2}\,$ and $z_q(\theta_0+\delta)$ converges to $z_1(\delta)=-i+\frac{2}{\sqrt{3}}e^{i\pi/6}\delta+o(\delta)\,$ for any small positive constant $\delta$. So, denoting $\alpha_q^+=\arg\{z_q(\theta_0+\delta)-z_q(\theta_0)\}\,$ and $\alpha_q^-=\pi-\alpha_q^+\,$, arguing as in subcase {\em (1.b)} we may choose $\delta\,,\,\kappa>0$ such that for $1-q_1$ small enough, there hold $\vert\alpha^+_q-\pi/6\vert<\pi/12\,$, $\vert\alpha^-_q-5\pi/6\vert<\pi/12\,$,
$\,\Re\{h''_q(z^0_\pm)e^{2i\alpha^\pm_q}\}\leq -\kappa\sqrt{1-q}\,$ and $\,\Re\{h^{(3)}_q(z^0_\pm)e^{3i\alpha^\pm_q}\}\leq -\kappa\,$
and, for some $\kappa'>0$ and all $t\in [0,\delta]\,$,
\begin{equation}\label{boundpm2}
\Re\{h_q(z^0_\pm \pm qt + iqt\tan \alpha_q^+)\}\leq q\,\ln\left(\frac{2q}{e}\right)-\kappa' \left(\sqrt{q-1}\,t^2+\vert t\vert^3\right)\,.
\end{equation}
The segments of ``controlled descent" $[z^0_\pm,z_q(\pm(\theta_0+\delta))]$
are going to replace the two pieces of curves $z_q(\pm [\theta_0,\theta_0+\delta])$.\medskip

We also need to modify $z_q(\pm (0,\theta_0])\,$. To do so, we choose $\beta^+=-3\pi/5$ and $\beta^-=-2\pi/5\,$ and we check that for $1-q_1$ small enough and some $\kappa''>0\,$,
$\,\Re\{h''_q(z^0_\pm)e^{2i\beta_\pm}\}\leq -\kappa''\sqrt{1-q}\,$ and $\,\Re\{h^{(3)}_q(z^0_\pm)e^{3i\beta_\pm}\}\leq -\kappa''\,.$
As a consequence, taking $1-q_1$ even smaller if necessary, we find that for some $\kappa'''>0$ and any $t\in [0,\tan\theta_0]\,$,
\begin{equation}\label{boundpm3}
\Re\{h_q(z^0_\pm \mp qt(1 + i\tan \beta_\pm)\}\leq q\,\ln\left(\frac{2q}{e}\right)-\kappa'''\left(\sqrt{q-1}\,t^2+\vert t\vert^3\right)\,.
\end{equation}
One has $z^0_\pm \mp q\tan\theta_0(1 + i\tan \beta_\pm)=-i\frac{\cos(\frac{2\pi}{5})}{\cos(\frac{2\pi}{5}-\theta_0)}\,$, so \eqref{boundpm3} means that the two segments $\Big[-i\frac{\cos(\frac{2\pi}{5})}{\cos(\frac{2\pi}{5}-\theta_0)}\,,\,z^0_\pm\Big]$ are of ``controlled descent". They are going to replace the unbounded branches of steepest descent $z_q(\pm(0,\theta_0])\,$.\medskip

\noindent
We thus define a new oriented curve connecting $-1$ to $1\,$:
\begin{equation}\label{Gammatildelast}
\tilde{\Gamma}_q:=\tilde{\Gamma}_q^{(-3)}\cup\tilde{\Gamma}_q^{(-2)}\cup\tilde{\Gamma}_q^{(-1)}\cup\tilde{\Gamma}_q^{(1)}\cup\tilde{\Gamma}_q^{(2)}\cup\tilde{\Gamma}_q^{(3)}\,,
\end{equation}
where, as shown in figure \ref{fig6},\medskip

\begin{eqnarray*}
&\tilde{\Gamma}_q^{(-3)}&=\big\{z_q(\theta)\,:\, -\theta_{\max}< \theta\leq- \theta_0-\delta \big\}\,;\\
&\tilde{\Gamma}_q^{(-2)}&=[z_q(-\theta_0-\delta)\,,\,z^0_-]\,;\\
&\tilde{\Gamma}_q^{(-1)}&=\left[z^0_-\,,\,-i\frac{\cos(\frac{2\pi}{5})}{\cos(\frac{2\pi}{5}-\theta_0)}\right]\,;\\
&\tilde{\Gamma}_q^{(1)}&=\left[-i\frac{\cos(\frac{2\pi}{5})}{\cos(\frac{2\pi}{5}-\theta_0)}\,,\,z^0_+\right]\,;\\
&\tilde{\Gamma}_q^{(2)}&=[z^0_+,z_q(\theta_0+\delta)]\,;\\
&\tilde{\Gamma}_q^{(3)}&=\big\{z_q(\theta)\,:\, \theta_0+\delta\leq \theta < \theta_{\max} \big\}\,.
\end{eqnarray*}

\begin{figure}[htp]
\begin{center}
\includegraphics[width=12cm,height=8cm]{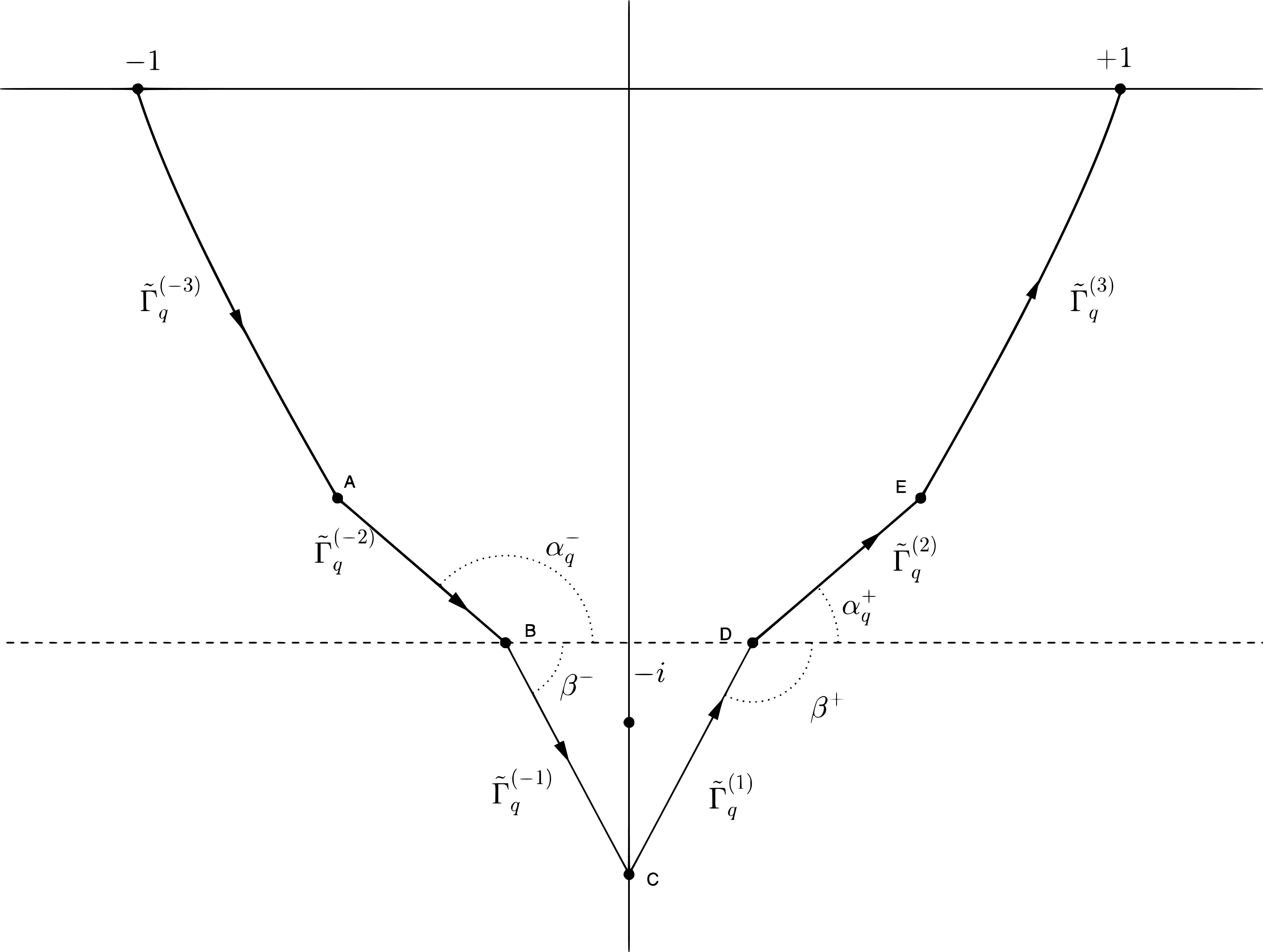}
\hspace{3mm}
\caption{\small{Picture of the curve $\tilde{\Gamma}_q$ defined by \eqref{Gammatildelast}.
}}
\label{fig6}
\end{center}
\end{figure}

{\em Highlighted points:}
\begin{itemize}

\item $A=z_q(-\theta_0-\delta)\,$, 
\item $B=z_-^0=z_q(-\theta_0)\,$, 
\item $C=-i\frac{\cos(\frac{2\pi}{5})}{\cos(\frac{2\pi}{5}-\theta_0)}\,$, 
\item $D=z^0_+=z_q(\theta_0)\,$, 
\item $E=z_q(\theta_0+\delta)\,$.
\end{itemize}

\medskip

\noindent
Then we follow the argument of subcase {\em (1.b)} and find estimates of the form
\begin{eqnarray*}
j_{\varepsilon,k}(\rho)&\leq& C \vert k\vert^{-1/2}\left( \int_\R e^{-\rho\min(\kappa',\kappa''') (\sqrt{q-1}\,t^2+|t|^3)}dt +  e^{-\rho\,\kappa' \delta^3}  \right)\\
&\leq &
C' \vert k\vert^{-\frac34}\left(\vert k\vert^{\frac13}+|\vert k\vert-\rho|\right)^{-\frac14}.
\end{eqnarray*}

{\em Subcase (2.c): $\displaystyle 0< q\leq q_2$ .} The main problem here is that for small $q$ the points $z^0_\pm$ are close to the singular points $\pm1$. Since $h_q$ and $\Gamma_q$ are symmetric with respect to the imaginary axis, we focus on $z^0_+$ and $\Gamma_q^r\,$. We note that $z^0_+-1\sim -iq $ when $q\to 0\,,$ so for a Taylor expansion of $\Re\{h_q(z)\}$ near $z^0_+=z_q(\theta_0)$ to have some validity, we need $\vert z-z^0_+\vert$ to be much smaller than $q$. But one can check that $z'_q(\theta_0)$ is close to $2i$ for $q$ small, so
this suggests to rescale the parametrization of $\Gamma_q$ around $\theta_0\,$.\medskip

After imposing the condition $q_2\leq 1/2$ we thus perform the local change of parameter $\theta(q,u)=\theta_0+qu\,\ (-u_0\leq u\leq u_0)\,$ where $u_0\in (0,1/4)$ is to be chosen later, and we denote $Z(q,u)=z_q(\theta(q,u))\,$. Then, denoting $X(q,u):=\Re\{Z(q,u)\}\,$ and $Y(q,u):=\Im\{Z(q,u)\}\,$, we have
\begin{equation*}
X(q,u)=\sqrt{1-q^{2}}+q^2 u,\quad Y(q,u)=-\varphi_q(\theta(q,u))\,.
\end{equation*}
We are now going to write expansions around the saddle point $z^0_+=Z(q,0)$ that keep track of the dependences on both $q$ and $u$.\medskip
 
 First of all, using the formulas $\sin(\theta(q,u))=\sqrt{1-q^2}\cos(qu)+q\sin(qu)$ and $\cos(\theta(q,u))=q\cos(qu)-\sqrt{1-q^2}\sin(qu)\,,$ then Taylor expanding $\cos(qu)$ and $\sin(qu)\,$ and injecting the expansions in \eqref{phieq}, we get after some straighforward calculations

\begin{equation}\label{yeq}
Y(q,u)=-q+q(\sqrt{1-q^{2}}+1)u+O\left(q^3 u^2\right)\,.
\end{equation}
We thus can write, after some further computations,
\begin{equation*}
|1-Z(q,u)^2|^2
=4q^2\left(1 -2(\sqrt{1-q^{2}}+1)u+4u^2+ O\left(q^2 u^2\right)\right)\,,
\end{equation*}
hence
\begin{equation*}
\Re\{h_q(Z(q,u))\}=Y(q,u)+q\ln|1-Z(q,u)^2|
=q\ln\left(\frac{2q}{e}\right) -2qu^2+ O\left(q|u|^3+q^2 u^2\right).
\end{equation*}
As a consequence, we obtain
\begin{equation}\label{estonR}
\Re\{h_q(Z(q,u))\}\leq q\ln\left(\frac{2q}{e}\right)-qu^2
\end{equation}
for 
$
|u|\leq u_0$ and $0<q\leq q_2\,$, if $u_0$ and $q_2$ are chosen small enough.\medskip

Now, from \eqref{phieq'}, there holds an estimate of the form 
$
\left|g_\varepsilon(Z(q,u))\frac{\partial Z}{\partial u}(q,u)\right|\leq C\,q
$
for $|u|\leq u_0$ and $0<q\leq q_2$. So, using
 \eqref{estonR}, we get
\begin{equation}\label{inter1}
\int_{z_q(\theta_0-qu_0,\theta_0+qu_0)}e^{\rho\,\Re\{h_q(z)\}}|g_\varepsilon(z)\vert\, \vert dz\vert \leq C q\left(\frac{2q}{e}\right)^{\gamma-1}(\gamma-1)^{-1/2}\,.
\end{equation}
On the other hand, for $\theta$ in $(0,\theta_{\max})\setminus [\theta_0-qu_0,\theta_0+qu_0]$, from \eqref{estonR} and the monotonocity of $\Re\{h_q\circ z_q(\theta)\}$ along gradient lines we have
$$e^{(\rho-1)\,\Re\{h_q\circ z_q(\theta)\}}\leq \left(\frac{2q}{e}\right)^{\gamma-1-q}e^{-(\gamma-1-q)\delta^2}$$
and we get from \eqref{phieq'}-\eqref{estim h} the estimate
$$e^{\Re\{h_q\circ z_q(\theta)\}}|g_\varepsilon(z_q(\theta))z_q'(\theta)| \leq C \theta^{-4}e^{-(1-q_2)^{3/2}a/\theta}\,.$$
Multipling these two bounds and integrating, we obtain
\begin{equation}\label{inter2}
\int_{\Gamma_q^r\setminus z_q(\theta_0-qu_0,\theta_0+qu_0)}e^{\rho\,\Re\{h_q(z)\}}|g_\varepsilon(z)\vert\, \vert dz\vert \leq C \left(\frac{2q}{e}\right)^{\gamma-1}e^{-\gamma\delta^2}\,.
\end{equation}
So, summing up the contributions of \eqref{inter1}, \eqref{inter2} and their analogues for the integrals on $z_q(-\theta_0-qu_0,-\theta_0+qu_0)$ and $\Gamma^l_q\setminus z_q(-\theta_0-qu_0,-\theta_0+qu_0)\,,$ we find
\begin{equation}\label{interfin}
\vert I_{\varepsilon,\gamma,\rho}\vert \leq C \left(\frac{2q}{e}\right)^{\gamma-1}\big(\,q(\gamma-1)^{-1/2}+e^{-\gamma\delta^2}\big)\,.
\end{equation}
Since we are in a sector such that $\rho\leq (\gamma+1)^2/2\,$, we have $e^{-\gamma\delta^2}=\mathcal O(q(\gamma-1)^{-1/2})\,$. So \eqref{interfin} gives
$$\vert I_{\varepsilon,\gamma,\rho}\vert \leq C' \left(\frac{2q}{e}\right)^{\gamma-1}\,q(\gamma-1)^{-1/2}\,$$
and finally, using \eqref{prefactor},
$$j_{\varepsilon,k}(\rho)\leq C''\rho^{-1}\,.$$
This ends the proof of Theorem \ref{esteigen}.


\medskip

\section{Strichartz Estimates}\label{strichsec}

Again, throughout this section the constant $C$ will be allowed to change from line to line; all that matters is that it can be taken independently of $k$, $R$ and $\nu$.

\subsection{A useful integral estimate}

The following integral bound on the generalized eigenfunctions is a consequence of Theorem \ref{esteigen}, and will play a crucial role in our proof of Theorem \ref{strichteo}. 

\begin{proposition}
Let $k\in\Z^*$, $\gamma=\sqrt{k^2-\nu^2}$ and let $\psi_k(r)$ be as in Theorem \ref{esteigen}. The following estimates hold
\begin{equation}\label{Fk}
\left(\int_R^{2R}|\psi_{k}(r)|^2 r^{2}dr\right)^{1/2}\leq C\times \begin{cases} R^{\gamma+1/2},\quad\; R\leq 1,\\
R^{1/2},\qquad\; R\geq 1\end{cases}
\end{equation}
and 
\begin{equation}\label{Fk'}
\left(\int_R^{2R}|\psi'_{k}(r)|^2 r^{2}dr\right)^{1/2}\leq C\times \begin{cases} R^{\gamma-1/2},\quad\;\;\; R\leq 1,\\
R^{1/2},\qquad\;\; \;R\geq 1,
\end{cases}
\end{equation}
where $C$ is a constant independent of $k$.
\end{proposition}

\begin{proof}
We separate the proof in two steps and we use the estimates \eqref{esteigformula}-\eqref{esteigformula'}.
\begin{itemize}
\item
If $R\leq 1$ we have
\begin{equation*}
\left(\int_R^{2R}|\psi_k(r)|^2 r^{2}dr\right)^{1/2}\leq C \left(\int_R^{2R}(r/2)^{2(\gamma-1)} r^{2}dr\right)^{1/2}\leq C R^{\gamma+1/2}
\end{equation*}
and
\begin{equation*}
\left(\int_R^{2R}|\psi'_{{k}}(r)|^2 r^{2}dr\right)^{1/2} \leq C \left(\int_R^{2R}(r/2)^{2(\gamma-2)} r^{2}dr\right)^{1/2}\leq C R^{\gamma-1/2}\,.
\end{equation*}
\item 
If $R\geq 1$, we deal separately with the following intervals:
$$[R,2R]=I_1+I_2+I_3$$
where 
$$I_1=[R,2R]\cap[0,\frac12|k|],\quad I_2=[R,2R]\cap[\frac12|k|, 2|k|], \quad I_3=[R,2R]\cap[2|k|,+\infty).$$
For what concerns $I_1$, we can assume that $2 R\leq |k|$, otherwise $I_1=\emptyset$. Then we have
\begin{equation*}
\begin{split}
\left(\int_{I_1}\big(|\psi_{{k}}(r)|^2+|\psi'_{k}(r)|^2\big) r^{2}dr\right)^{1/2}&\leq C \left(\int_R^{2R}e^{-2D|k|} r^{2}dr\right)^{1/2}\\
&\leq C \left(\int_R^{2R}e^{-2Dr} r^{2}dr\right)^{1/2}\\
&\leq CR^{1/2}\,.
\end{split}
\end{equation*}
For $I_2$, we can assume that $|k|/4\leq R\leq 2 |k|$, otherwise $I_2=\emptyset$. Then we get, with simple computations, 
\begin{equation*}
\begin{split}
\left(\int_{I_2}\big(|\psi_{{k}}(r)|^2+|\psi'_{k}(r)|^2\big) r^{2}dr\right)^{1/2}&\leq C|k|^{-\frac34}\left(\int_{\frac{|k|}2}^{2|k|}\big(||k|-r|+|k|^{\frac13}\big)^{-\frac12} r^{2}dr\right)^{1/2}\\&
\leq C|k|^{\frac14}\left(\int_{\frac{|k|}2}^{2|k|}||k|-r|^{-\frac12} dr\right)^{1/2}\\&
\leq C|k|^{\frac12}
\leq CR^{\frac12}\,.
\end{split}
\end{equation*}

Finally, for what concerns $I_3$, we have 
\begin{equation*}
\left(\int_{I_3}\big(|\psi_{{k}}(r)|^2+|\psi'_{k}(r)|^2\big) r^{2}dr\right)^{1/2}\leq C \left(\int_R^{2R}(r^{-1})^2r^{2}dr\right)^{1/2}
\leq CR^{1/2}
\end{equation*}
and this concludes the proof.
\end{itemize}
\end{proof}

\subsection{Proof of Theorem \ref{strichteo}}

As mentioned in Remark \ref{strichartzrange}, we are going to prove the Strichartz estimate on the ``endpoint board line", that is
\begin{equation}\label{stri}
\|\sol\|_{L^2_tL^q_{r^2dr}L^2_\omega}\leq
C
\|u_0\|_{\dot{H}^{s}},\qquad  s=1-\frac3q
\end{equation}
under the condition 
\begin{equation}\label{q-cond}
4<q<\frac{3}{1-\sqrt{1-\nu^2}}.
\end{equation}
Then, by interpolation with the standard $L^\infty H^s$ estimate, we will be able to cover the range given by \eqref{q-condbis}.
\medskip

Our starting point is formula \eqref{repsol}: with this representation, we write, for $q \geq2$, thanks to the $L^2$-orthogonality of spherical harmonics, 
\begin{eqnarray}\label{stim1}
\nonumber
\|\sol\|_{L^2_tL^q_{r^{2}dr}L^2_\omega}&=&\|\sum_{\substack{k\in \mathbb{Z}^*\\ m\in \mathcal{I}_k}}(e^{-it\mathcal{D}_{\nu, k}}f_{0,k,m})\cdot {\Xi}_{ k,m}
\|_{L^2_tL^q_{r^{2}dr}L^2_\omega}\\
&\leq&
\left(\sum_{\substack{k\in \mathbb{Z}^*\\ m\in \mathcal{I}_k}}\|  \PP^{-1}\left[e^{-it\rho\sigma_3}\big(\PP f_{0,k,m}\big)(\rho)\right]\|^2_{L^2_tL^q_{r^{2}dr}}\right)^{1/2}.
\end{eqnarray}
From now on, in order to simplify the notations as much as we can, we shall omit the dependence on $m$ in the sum. Also, we shall develop the computations only for the positive part of the spectrum, that is for the first component in representation \eqref{repsol}, as the estimates in the other case are completely analogous. 

We start by proving the following Strichartz estimates with unit frequency

\begin{proposition}\label{pro}
Let  
$4<q<\frac{3}{1-\sqrt{1-\nu^2}}$ and assume
$\text{supp}\big(\PP f_{0,k}\big)\subset [1,2]$
for all $k\in \mathbb{Z}^*$. Then
\begin{equation}\label{4.3}
\begin{split}
\|\sol\|_{L^2_tL^q_{r^{2}dr}L^2_\omega}\leq
C\| u_0\|_{ L^2_x}.
\end{split}
\end{equation}
\end{proposition}

\begin{remark}
The lower bound $4$ in the range of $q$ comes from the bound $\frac{2(n-1)}{n-2}$ in the cone restriction estimate when $n=3$, see Tao \cite{Tao}; notice that we retrieve the same bound from below through our strategy.
\end{remark}

\begin{proof}

Let us denote $\PP f_{0,k}=g_{k}$; then, due to \eqref{stim1}, it suffices to show
\begin{equation}\label{4.10}
\begin{split}
\sum_{k\in \mathbb{Z}^*}\|  \PP^{-1}\left[e^{-it\rho\sigma_3}g_{k}(\rho)\right]\|^2_{L^2_tL^q_{r^{2}dr}}\leq
C\|u_0\|^2_{L^2_x}.
\end{split}
\end{equation}
Using the dyadic decomposition, we have by $\ell^{2}\hookrightarrow
\ell^{q}\,(q\geq2)$
\begin{equation}\label{4.11}
\begin{split}
&\sum_{k\in \mathbb{Z}^*}\|  \PP^{-1}\left[e^{-it\rho\sigma_3}g_{k}(\rho)\right]\|^2_{L^2_tL^q_{r^{2}dr}}\\&=
\sum_{k\in \mathbb{Z}^*}\Big\|\Big(\sum_{R\in2^{\Z}}\Big\|\PP^{-1}\big[e^{-it\rho\sigma_3}g_k(\rho)\big](r)\Big\|^q_{L^q_{r^{2}dr}([R,2R])}\Big)^{\frac1q}\Big\|^2_{L^2_t}
\\&\leq
\sum_{R\in2^{\Z}}\sum_{k\in\Z^*}\Big\|\PP^{-1}\big[e^{-it\rho\sigma_3}g_k(\rho)\big](r)\Big\|^2_{L^2_tL^q_{r^{2}dr}([R,2R])}.
\end{split}
\end{equation}

We need the following result

\begin{lemma}\label{Hankel2}Let $q\geq2$ and $k\in \Z^*$.
Suppose $\text{supp}(g_k)\subset I:=[1,2]$.  Then
\begin{equation}\label{est:hankel}
\begin{split}
&\Big\|\PP^{-1}\big[e^{-it\rho\sigma_3}g_k(\rho)\big](r)\Big\|_{L^2_tL^q_{r^{2}dr}([R,2R])}
\\&\leq C
\|g_k(\rho)\|_{L^2_{\rho^2 d\rho}}\times \begin{cases}R^{\sqrt{k^2-\nu^2}-1+\frac{3}q},\quad \;R\leq 1 \\
R^{\frac{2}q-\frac12}\qquad\qquad R\geq1
\end{cases}
\end{split}
\end{equation}
where $R\in 2^{\Z}$ and $C$ is a constant independent of $R$ and $k$.
\end{lemma}

\begin{proof} Recalling that $\sigma_3$ is given by \eqref{sigma}, to simplify the notations we replace $e^{-it\rho\sigma_3}$ by $e^{it\rho}$ and we forget the fact that $\psi_k(r\rho)$ and $g_k(\rho)$ are vectors, we consider them as scalars. It is to be understood that we work with their coordinates. Also, let us stress the fact that, as $g_k$ is supported in $[1,2]$, all norms of the form $\|g_k\|_{L^2_{\rho^\alpha d\rho}}$ are equivalent; this fact will be used several times throughout the proof.
We  first consider the case $R\geq 1$. We need to prove the following estimates for a fixed $k\in\Z^*$:
\begin{equation}\label{p-2}
\begin{split}
\Big\|\int_0^\infty e^{it\rho} {\psi}_{k}(
r\rho)g_{k}(\rho) \rho^{2}d\rho\Big\|_{L^2_tL^2_{r^{2}dr}([R,2R])}&
\leq C R^{\frac1 2}\| g_{k}(\rho)\|_{L^{2}_{\rho^2 d\rho}},
\end{split}
\end{equation}
and
\begin{equation}\label{p-in}
\begin{split}
\Big\|\int_0^\infty e^{it\rho} {\psi}_{k}(
r\rho)g_{k}(\rho)&\rho^{2}d\rho\Big\|_{L^2_tL^\infty_{r^{2}dr}([R,2R])}
\leq C R^{-\frac12}\| g_{k}(\rho)\|_{L^{2}_{\rho^2 d\rho}}.
\end{split}
\end{equation}
 To prove \eqref{p-2}, we use the Plancherel inequality in the time variable $t$ and \eqref{Fk} to obtain 
\begin{equation*}
\begin{split}
\text{L.H.S of }\quad\eqref{p-2}\lesssim \Big\| \big\|\psi_k( r\rho)
g_{k}(\rho) \big\|_{L^{2}_{ \rho^{4} d\rho}}
\Big\|_{L^2_{r^{2}dr}([R,2R])} \leq C R^{1/2}\|g_{k}(\rho)\|_{L^{2}_{\rho^2 d\rho}}.
\end{split}
\end{equation*}

We now prove \eqref{p-in}, which is the same as 
\begin{equation}
\begin{split}
\Big\|\int_0^\infty e^{it\rho} {\psi}_{k}(
r\rho)g_{k}(\rho)&\rho^{2}d\rho\Big\|_{L^2_tL^\infty_{dr}([R,2R])}
\leq C R^{-\frac12}\| g_{k}(\rho)\rho^2 \|_{L^{2}_{\rho^2 d\rho}}.
\end{split}
\end{equation}

By the Sobolev embedding $H^1(\Omega)\hookrightarrow
L^\infty(\Omega)$ with $\Omega=[R,2R]$ and \eqref{p-2}, it suffices to show
\begin{equation}\label{2-in}
\begin{split}
\Big\|\int_0^\infty e^{it\rho}  \psi^{\prime}_{k}(r\rho)
g_{k}(\rho) \rho^3 d\rho
\Big\|_{L^2_tL^2_{dr}([R,2R])}\leq C R^{-\frac12}
\|g_{k}(\rho)\|_{L^{2}_{\rho^2 d\rho}}.
\end{split}
\end{equation}
As in the above, by applying the Plancherel Theorem in $t$ and \eqref{Fk'} we obtain
\begin{equation*}
\begin{split}
\text{L.H.S of }~\eqref{2-in}\lesssim \Big\| \big\|\psi'_{k}( r\rho)
g_{k}(\rho) \rho^3 \big\|_{L^{2}_{d\rho}}
\Big\|_{L^2_{dr}([R,2R])}\leq C R^{-1/2} \|g_{k}(\rho)\|_{L^{2}_{\rho^2 d\rho}}.
\end{split}
\end{equation*}

Secondly, we consider the case $R\leq 1$.
By the Sobolev embedding $H^{\frac12-\frac1q}(\Omega)\hookrightarrow
L^q(\Omega)$ and interpolation, we have
\begin{equation}\label{3.14}
\begin{split}
&\Big\|\int_0^\infty e^{it\rho} {\psi}_{k}(
r\rho)g_{k}(\rho)
\rho^{2}d\rho\Big\|_{L^2_tL^q_{dr}([R,2R])}
\\&\leq \Big\|\int_0^\infty e^{it\rho} {\psi}_{k}(
r\rho)g_{k}(\rho)
\rho^{2}d\rho\Big\|^{\frac12+\frac1q}_{L^2_tL^{2}([R,2R])} \\&\qquad\times\Big\|\int_0^\infty e^{it\rho} {\psi}_{k}(
r\rho)g_{k}(\rho)
\rho^{2}d\rho\Big\|^{\frac12-\frac1q}_{L^2_t H^{1}([R,2R])}
\\&\leq C R^{\gamma-\frac1 2-(\frac12-\frac{1}q)}\|
g_{k}(\rho)\|_{L^{2}_{\rho^2 d\rho}},
\end{split}
\end{equation}
since from \eqref{Fk}we have 
\begin{equation*}
\begin{split}
\Big\|\int_0^\infty e^{it\rho} {\psi}_{k}(
r\rho)g_{k}(\rho)
&\rho^{2}d\rho\Big\|_{L^2_tL^2_{dr}([R,2R])}=\Big\| \big\|\psi_k( r\rho)
g_{k}(\rho) \big\|_{L^{2}_\rho}\Big\|_{L^2_{dr}([R,2R])}\\&
\leq C{ R^{\gamma-\frac1 2}}\| g_{k}(\rho)\|_{L^{2}_{\rho^2 d\rho}},
\end{split}
\end{equation*}
and from \eqref{Fk'} we have
\begin{equation*}
\begin{split}
\Big\|\int_0^\infty e^{it\rho} {\psi}_{k}'(
r\rho)g_{k}(\rho)
&\rho^{3} d\rho\Big\|_{L^2_tL^2_{dr}([R,2R])}=\Big\| \big\|\psi'_{k}( r\rho)
g_{k}(\rho)\big\|_{L^{2}_\rho}\Big\|_{L^2_{dr}([R,2R])}\\&
\leq C{R^{\gamma-\frac3 2}}\| g_{k}(\rho)\|_{L^{2}_{\rho^2 d\rho}}.
\end{split}
\end{equation*}
We have thus obtained
\begin{equation}
\begin{split}
\Big\|\int_0^\infty e^{it\rho} {\psi}_{k}(
r\rho)g_{k}(\rho)
\rho^{2}d\rho\Big\|_{L^2_tL^q_{r^{2}dr}([R,2R])}
\leq C R^{\gamma-1+\frac{3}q}\|
g_{k}(\rho)\|_{L^{2}_{\rho^2 d\rho}},
\end{split}
\end{equation}
and this concludes the proof of Lemma \ref{Hankel2}.

\end{proof}

We are now in position for proving Proposition \ref{pro}. Thanks to estimate \eqref{est:hankel}
 we get
\begin{equation}\label{4.12}
\begin{split}
&\sum_{R\in2^{\Z}}\sum_{k\in\Z^*}\Big\|\PP^{-1}\big[e^{it\rho\sigma_3}g_k(\rho)\big](r)\Big\|^2_{L^2_tL^q_{r^{2}dr}([R,2R])}
\\&\leq C
\sum_{k\in\Z^*}\Big(\sum_{R\in2^{\Z}:R\leq 1}{R^{2\big(\sqrt{k^2-\nu^2}-1+\frac{3}q\big)}}+\sum_{R\in2^{\Z}: R\geq1} 2R^{2(\frac{2}q-\frac{1}2)}\Big)\|
g_k(\rho)\|^2_{L^{2}_{\rho^2 d\rho}}\\&\leq C
\sum_{k\in\Z^*}\|
g_k(\rho)\|^2_{L^{2}_{\rho^2 d\rho}}, \qquad 
\end{split}
\end{equation}
provided that, for any $k\in\mathbb{Z}^*$,
\begin{equation}\label{condK}
\sqrt{k^2-\nu^2}-1+\frac{3}q>0\quad\hbox{and} \quad\frac{2}q-\frac{1}2<0\,,
\end{equation}
that is, for general data, provided 
$$4<q<\frac{3}{1-\sqrt{ 1-\nu^2}}.$$
We stress the fact that the first condition in \eqref{condK} is automatically satisfied if $|k|\geq2$, that is, if the initial datum does not have a component in the ``first partial wave subspaces", {\it i.e.} the ones corresponding to the case $k=\pm1$. Thus, in this case, there is no upper bound on $q$.\smallskip

\noindent
We recall that by definition, $g_k=\PP f_{0,k}$.
As $\PP$ is isometric on $L^2$, \begin{equation*}
\begin{split}
\sum_{k\in\Z^*}\|
\big(\PP f_{0,k}\big)(\rho)\|^2_{L^2_{\rho^{2}d\rho}}=
\sum_{k\in\Z^*}\|
f_{0,k}(r)\|^2_{L^2_{r^{2}dr}}=\|u_0\|^2_{L^2_x}.
\end{split}
\end{equation*}
This concludes the proof of Proposition \ref{pro}.
\end{proof}

We are finally in position for proving Theorem \ref{strichteo}. Let $R$ and $N$ be dyadic numbers (i.e. let $R$ and $N$ be in $2^{\Z}$) and let $\beta\in C_c^\infty([1,2])$; by making a dyadic decomposition, and using a scaling argument we can write, starting from \eqref{stim1},
\begin{equation*}
\begin{split}
&\sum_{\substack{k\in \mathbb{Z}^*\\ m\in \mathcal{I}_k}}\|  \PP^{-1}\left[e^{-it\rho\sigma_3}\big(\PP f_{0,k,m}\big)(\rho)\right]\|^2_{L^2_tL^q_{r^{2}dr}}\\ &\leq C\sum_{k\in\Z^*}\Big\|\sum_{N\in2^\Z}\mathcal{P}_{k}^{-1}\left[e^{-it\rho\sigma_3} \beta(\frac{\rho}N)[\mathcal{P}_{k}f_{k}](\rho)\right]\Big\|^2_{L^2_tL^q_{r^{2}dr}} \\
&\leq C\sum_{k\in\Z^*}\sum_{R\in2^\Z}\Big\|  \sum_{N\in2^\Z}\mathcal{P}_{k}^{-1}\left[e^{-it\rho\sigma_3} \beta(\frac{\rho}N)g_{k}(\rho)\right]\Big\|^2_{L^2_tL^q_{r^{2}dr}([R,2R])} \\
&\leq C\sum_{k\in\Z^*}\sum_{R\in2^\Z} \Big( \sum_{N\in2^\Z}\Big\| \mathcal{P}_{k}^{-1}\left[e^{-it\rho\sigma_3} \beta(\frac{\rho}N)g_{k}(\rho)\right]\Big\|_{L^2_tL^q_{r^{2}dr}([R,2R])}\Big)^2\\
&= \sum_{k\in\Z^*}\sum_{R\in2^\Z} \Big( \sum_{N\in2^\Z}N^{3(1-\frac1q)-\frac12}\Big\| \mathcal{P}_{k}^{-1}\left[e^{-it\rho\sigma_3} \beta(\rho)g_{k}(N\rho)\right]\Big\|_{L^2_tL^q_{r^{2}dr}([NR,2NR])}\Big)^2.
\end{split}
\end{equation*}
At this point we are in position to exploit Proposition \ref{pro} (and then re-scale again): we can thus estimate further with
 \begin{equation}
\begin{split}
&\leq \sum_{k\in\Z^*}\sum_{R\in2^\Z}  \Big(\sum_{N\in2^\Z}N^{3(\frac12 -\frac 3q)-\frac12} Q(NR)
\|\beta(\rho/N)g_{k}(\rho)\|_{L^2_{\rho^{2}d\rho}}
\Big)^2\\
\end{split}
\end{equation}
where 
\begin{equation*}
Q(NR)=
\begin{cases}{(NR)^{\gamma-1+\frac{3}q},\qquad\;\;\; NR\leq 1}, \\
(NR)^{\frac{2}q-\frac12}\qquad\qquad NR\geq 1.
\end{cases}
\end{equation*}
Due to \eqref{q-cond} and $\gamma=\sqrt{k^2-\nu^2}$,
we note that
\begin{equation}\label{ST}
\sup_{R} \sum_{N\in2^\Z} Q(NR) <\infty,\quad \sup_{N} \sum_{R\in2^\Z} Q(NR) <\infty.
\end{equation}
Let 
$A_{N,k}=N^{3(\frac12 -\frac 1q)-\frac12}
\|g_{k}(\rho)\beta(\rho/N)\|_{L^2_{\rho^{2}d\rho}(\R^+)}$,
we use the Schur test Lemma with \eqref{ST} in the following way: 
\begin{equation*}
\begin{split}
  \left(\sum_{R\in2^\Z}  \Big(\sum_{N\in2^\Z} Q(NR) A_{N,k} \Big)^2\right)^{1/2}
 =\sup_{\|B_R\|_{\ell^2}\leq 1}\sum_{R\in2^\Z}  \sum_{N\in2^\Z} Q(NR)A_{N,k} B_R
 \end{split}
\end{equation*}
which is bounded by
 \begin{equation*}
\begin{split}
 &\leq C  \left(\sum_{R\in2^\Z}\sum_{N\in2^\Z} Q(NR)|A_{N,k}|^2\right)^{1/2} \left(\sum_{R\in2^\Z}\sum_{N\in2^\Z} Q(NR)|B_R|^2\right)^{1/2}
 \\&\leq C\big( \sup_{R} \sum_{N\in2^\Z} Q(NR) \sup_{N} \sum_{R\in2^\Z} Q(NR)\big)^{1/2} \left(\sum_{N\in2^\Z}|A_{N,k}|^2\right)^{1/2} \left(\sum_{R\in2^\Z}|B_R|^2\right)^{1/2}
 \\&\leq C\left(\sum_{N\in2^\Z}|A_{N,k}|^2\right)^{1/2} .
\end{split}
\end{equation*}
We have thus obtained, recalling the definition of $g_k=\PP f_{0,k}$, and the properties of our Hankel transform given in Proposition \ref{properties},
\begin{eqnarray*}
\|e^{-it\mathcal{D}_\nu}u_{0}\|_{L^2_t(\R;L^q_{r^2dr}L^2_\omega)}^2 &\leq&\sum_{k\in\Z^*}\sum_{R\in2^\Z}  \Big(\sum_{N\in2^\Z} Q(NR) A_{N,k} \Big)^2
\\
&\leq& C\sum_{k\in\Z^*}\sum_{N\in2^\Z}|A_{N,k}|^2\\
&=& C\sum_{k\in\Z^*}\sum_{N\in2^\Z} N^{2(3(\frac12 -\frac 1q)-\frac12)}
\|g_{k}(\rho)\beta(\rho/N)\|^2_{L^2_{\rho^{2}d\rho}(\R^+)}\\
&=& C\sum_{k\in\Z^*}\sum_{N\in2^\Z} N^{2(3(\frac12 -\frac 1q)-\frac12)}
\|\PP f_{0,k}(\rho) \beta(\rho/N)\|^2_{L^2_{\rho^{2} d\rho}(\R^+)}
\\
&=& C\sum_{k\in\Z^*}\sum_{N\in2^\Z} N^{2(3(\frac12 -\frac 1q)-\frac12)}
\| \beta({|\mathcal{D}_\nu|}/N)f_{0,k}(r)\|^2_{L^2_{r^{2} dr}(\R^+)}\\&
\leq& C\||\mathcal{D}_\nu|^{3(\frac12 -\frac 1q)-\frac12)} u_0\|_{L^2}^2,
\\&
\leq& C\| u_0\|_{\dot H^{1-\frac3q}}^2,
\end{eqnarray*}
where in the last inequality we have used Lemma \ref{confnorm} to estimate the fractional powers of the Dirac-Coulomb operator with standard fractional derivatives, and the proof of Theorem \ref{strichteo} is concluded. $\square$

\bigskip\bigskip\bigskip\bigskip

{\bf Acknowledgments.}  
The authors are grateful to Konstantin Merz and Marta Strani for providing helpful comments. FC acknowledges support from the University of Padova STARS project ``Linear and Nonlinear Problems for the Dirac Equation" (LANPDE); JZ acknowledges support from National Natural Science Foundation of China (12171031,11771041, 11831004). The authors are also grateful to the anonymous referees for their careful reading of the paper.


\begin{thebibliography}{9}

\bibitem{abram}
M. Abramowitz and I. A  Stegun.
\newblock Handbook of mathematical functions with formulas,
graphs, and mathematical tables. 
\newblock National Bureau of Standards Applied Mathematics Series 55 (1964).

\bibitem{barcor}
J. A. Barcel\'{o} and A. C\'{o}rdoba.
\newblock Band-limited functions: $L^p$-convergence, 
\newblock {\em Trans. Amer. Math. Soc.} 313 , 655-669 (1989).


\bibitem{bousdanfan}
N. Boussaid, P. D'Ancona and L. Fanelli.
\newblock Virial identity and weak dispersion for the magnetic Dirac equation. 
\newblock {\em J. Math. Pures Appl.} (9) 95, no. 2, 137–150 (2011).

\bibitem{burq1}
N. Burq, F. Planchon, J.G. Stalker and A. Tahvildar-Zadeh Shadi.
\newblock Strichartz estimates for the wave and Schr\"odinger equations with the inverse-square potential. 
\newblock {\em J. Funct. Anal.} 203 (2), 519--549 (2003). 


\bibitem{cac1}
F. Cacciafesta.
\newblock Global small solutions to the critical Dirac equation with potential.
\newblock {\em Nonlinear Analysis} 74, pp. 6060-6073, (2011).  

\bibitem{cacdan}
F. Cacciafesta and P. D'Ancona:
\newblock Endpoint estimates and global existence for the nonlinear Dirac equation with potential.
\newblock {\em J. Differential Equations} 254 2233-2260 (2013).


\bibitem{cacfan}
F. Cacciafesta and L. Fanelli.
\newblock Dispersive estimates for the Dirac equation in an Aharonov-Bohm field.
\newblock {\em J. Differential equations} 263 7,  4382-4399, (2017).

\bibitem{cacser}
F. Cacciafesta and \'E. S\'er\'e.
\newblock Local smoothing estimates for the Dirac Coulomb equation in 2 and 3 dimensions.
\newblock {\em J. Funct. Anal.} 271  no.8, 2339-2358 (2016).

\bibitem{caczhaab}
F. Cacciafesta, Z. Yin and J. Zhang.
\newblock Generalized Strichartz estimates for wave and Dirac equation in Aharonov-Bohm magnetic fields,
\newblock {Dynamics of PDE} 19(1), 71--90, (2022).

\bibitem{chester}
C. Chester, B. Friedman, and F. Ursell.
\newblock An extension of the method
of steepest descents,
\newblock {Math. Proc. Cambridge} 53(3), 599--611 (1957).

\bibitem{cor}
A. C\'{o}rdoba.
\newblock The disc multiplier,
\newblock {\em Duke Math. J.} 58 , 21-29 (1989).

\bibitem{danfan}
 P. D'Ancona and L. Fanelli.
 \newblock Strichartz and smoothing estimates of dispersive equations with magnetic potentials. 
 \newblock {\em Comm. Partial Differential Equations} 33, no. 4-6, 1082–1112 (2008).

\bibitem{dieu}
J. Dieudonn\'{e}.
\newblock Calcul infinit\'{e}simal
\newblock Hermann \'{e}diteurs des sciences et des arts, (1980).

\bibitem{erd}
 A. Erd\'elyi.
\newblock Asymptotic expansions. 
\newblock Dover Publications, Inc., New York, (1956). 
 
 \bibitem{erdgreen}
  M. B. Erdogan, M. Goldberg and W. R. Green.
  \newblock Limiting absorption principle and Strichartz estimates for Dirac operators in two and higher dimensions. 
  \newblock {\em Comm. Math. Phys.} 367 , no. 1, 241–263, (2019).
  
 \bibitem{erdgreen2}
 M. B. Erdogan, W. R. Green and E. Toprak.
 \newblock Dispersive estimates for Dirac operators in dimension three with obstructions at threshold energies. 
 \newblock {\em Amer. J. Math.} 141 no. 5, 1217–1258, (2019).
 
\bibitem{estlewser}
M. J. Esteban, M. Lewin, and \'E. S\'er\'e.
 Domains for Dirac-Coulomb min-max levels. 
 \newblock {\em Rev. Mat. Iberoam.}, 35(3):877-924, (2019). 
 
\bibitem{frank}
R. L. Frank, K. Merz and H. Siedentop.
\newblock Equivalence of Sobolev norms involving generalized Hardy operators.   
\newblock  {\em Int. Math. Res. Not. IMRN}, no. 3, 2284–2303 (2021).

\bibitem{jiang}
J. C. Jiang, C. Wang, and X. Yu. 
\newblock Generalized and weighted Strichartz estimates. 
\newblock {\em Commun. Pure Appl. Anal. } 11 no. 5, 1723-1752 (2012).

\bibitem{landlif2}
L.D. Landau and L.M. Lifshitz.
\newblock{Quantum mechanics - Relativistic quantum theory} Pergamon Press. First edition 1971.

\bibitem{mach}
S. Machihara, M. Nakamura, K. Nakanishi, and T. Ozawa.
\newblock Endpoint Strichartz estimates and global solutions for the nonlinear Dirac equation.
\newblock {\em J. Funct. Anal.}, 219 (1):1-20, (2005).

\bibitem{miao}
C. Miao, J. Zhang and J. Zheng,
 \newblock Strichartz estimates for wave equation with inverse square potential,
 \newblock {\em Commun. Contemp. Math.} 15, no. 6, 1350026   (2013).

\bibitem{stempak}
\newblock K. Stempak.
\newblock A weighted uniform $L^p$-estimate of Bessel functions: a note on a paper by Guo.
\newblock {\em Proc. Am. Math. Soc.} 128 10, Pages 2943-2945 (2000).

\bibitem{sterbenz}
 J. Sterbenz.
 \newblock Angular regularity and Strichartz estimates for the wave equation (with an appendix by Igor Rodnianski)
 \newblock {\em Int. Math. Res. Not.} no. 4, 187-231, (2005).
 
\bibitem{Tao00}
T. Tao.
\newblock Spherically averaged endpoint Strichartz estimates for the two-dimensional Schrödinger equation. 
\newblock {\em Comm. Partial Differential Equations} 25, no. 7-8, 1471-1485 (2000).


\bibitem{Tao}
T. Tao,
\newblock Some recent progress on the restriction conjecture. In: Brandolini L., Colzani L., Travaglini G., Iosevich A. (eds) Fourier Analysis and Convexity. Applied and Numerical Harmonic Analysis.
\newblock {\em Birkh\"auser, Boston, MA.} 

\bibitem{temme}
N. M. Temme.
\newblock Uniform asymptotic methods for integrals.
\newblock {\em Indag. Math.} 24, 739-765 (2013).


\bibitem{thaller}
B. Thaller.
\newblock{The Dirac Equation.}
\newblock{ Springer-Verlag}, Texts and Monographs in Physics (1992).

\end{thebibliography}
\end{document}